\documentclass{article}
\usepackage{pinlabel}
\usepackage{graphicx}
\usepackage{amsmath}
\usepackage{amssymb}
\usepackage{amsthm}
\usepackage{multirow}

\newcommand{\pic}[2]{\raisebox{-.5\height}
{\includegraphics[scale=#2]{#1}}}

\def\nestedcircles{\pic{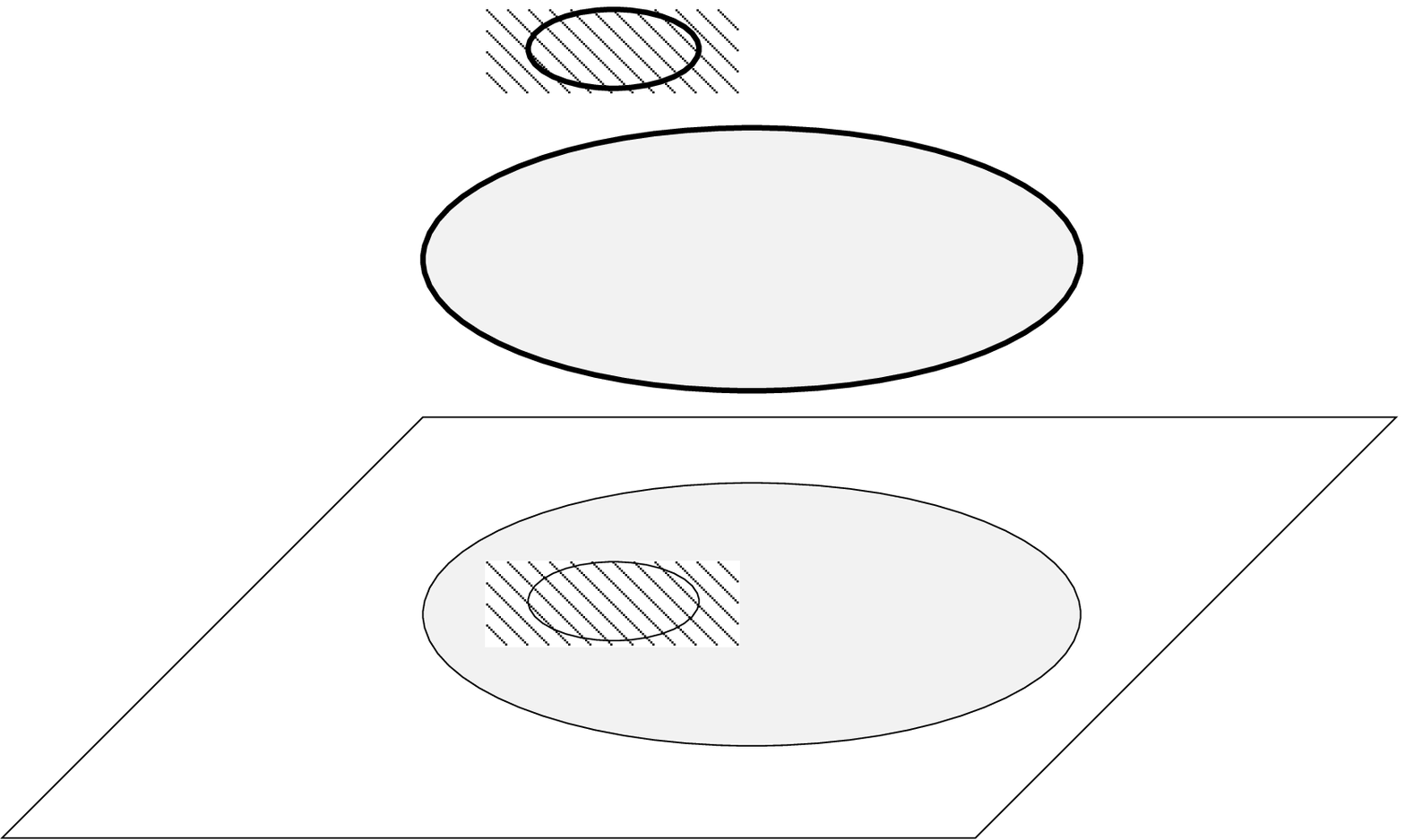}{.200}}
\def\NoPossible{\pic{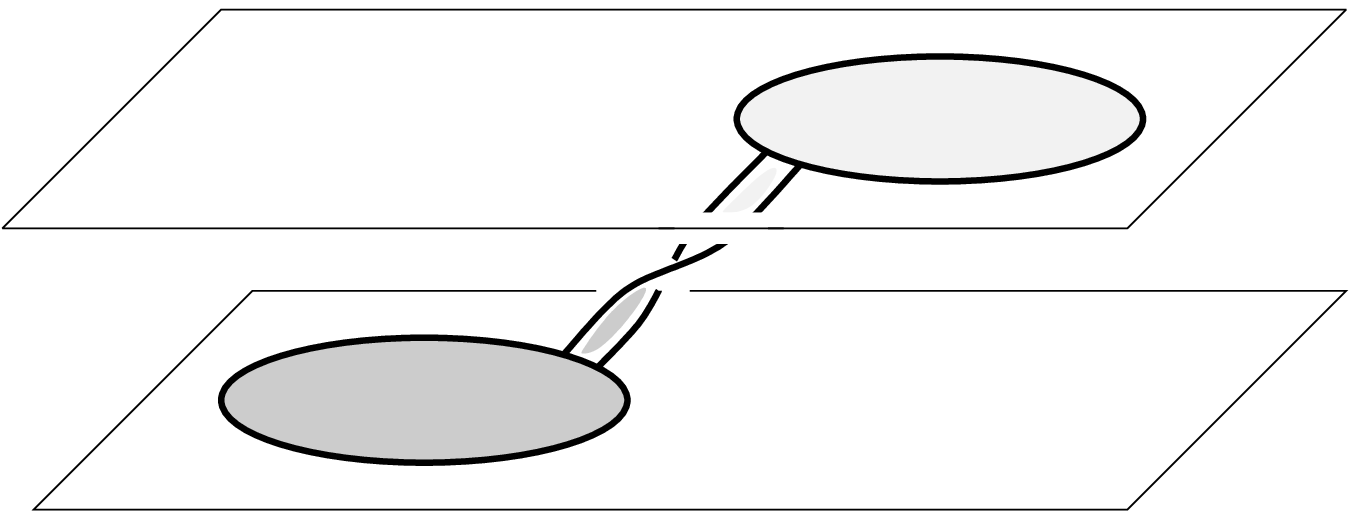}{.250}}
\def\friedeggs{\pic{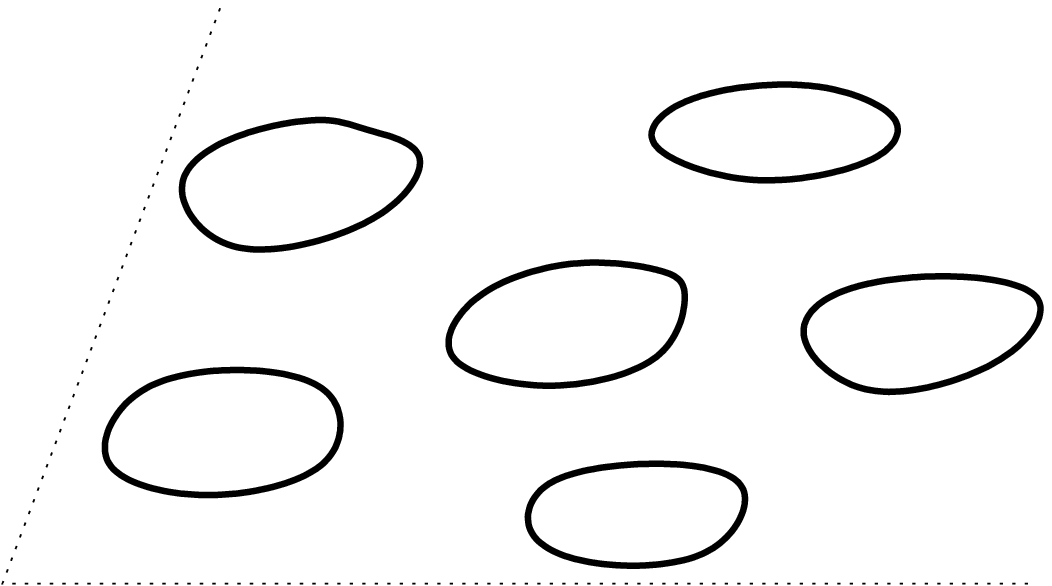}{.150}}
\def\friedeggswithpan{\pic{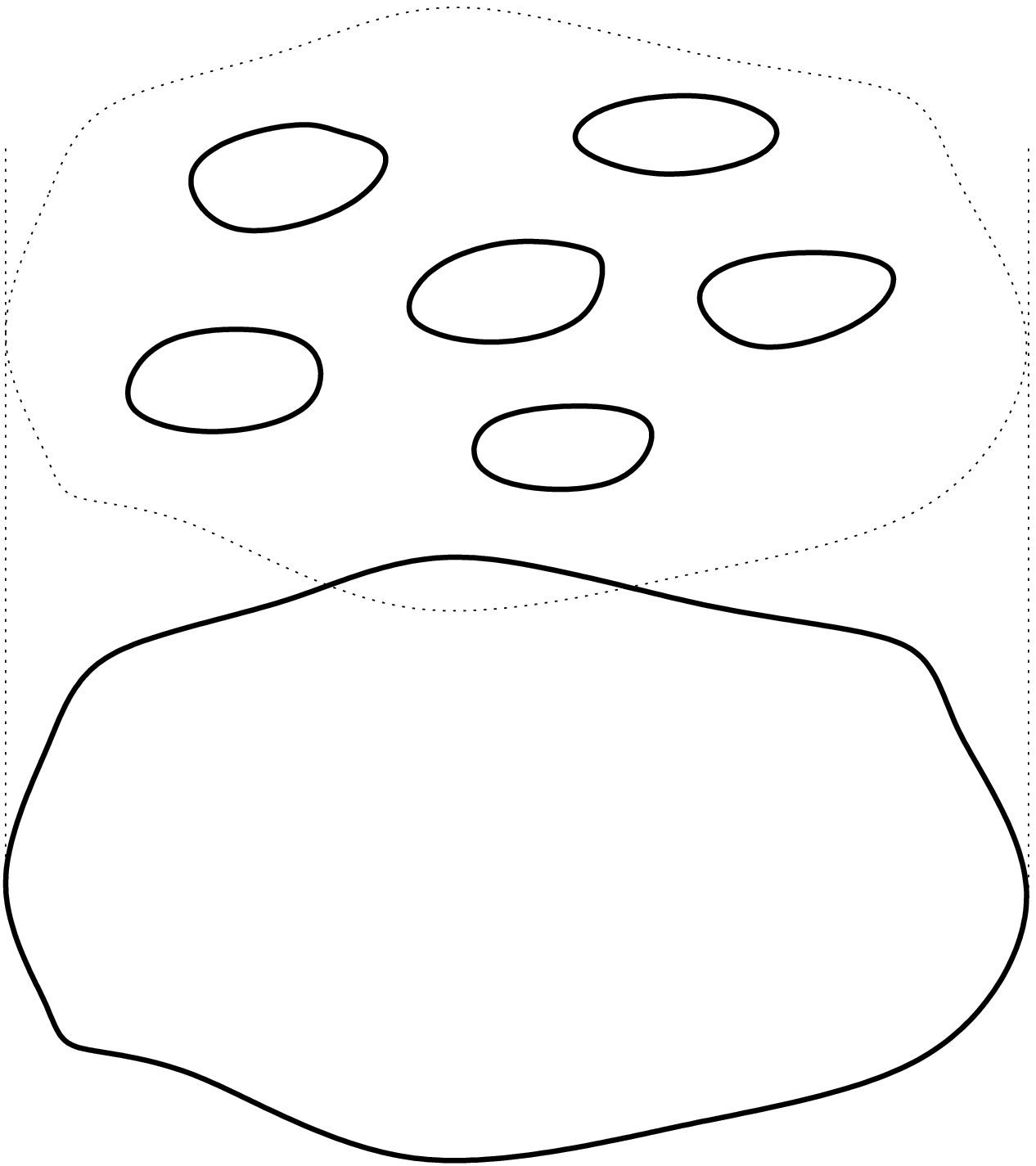}{.150}}
\def\BIBiIIDegg{\pic{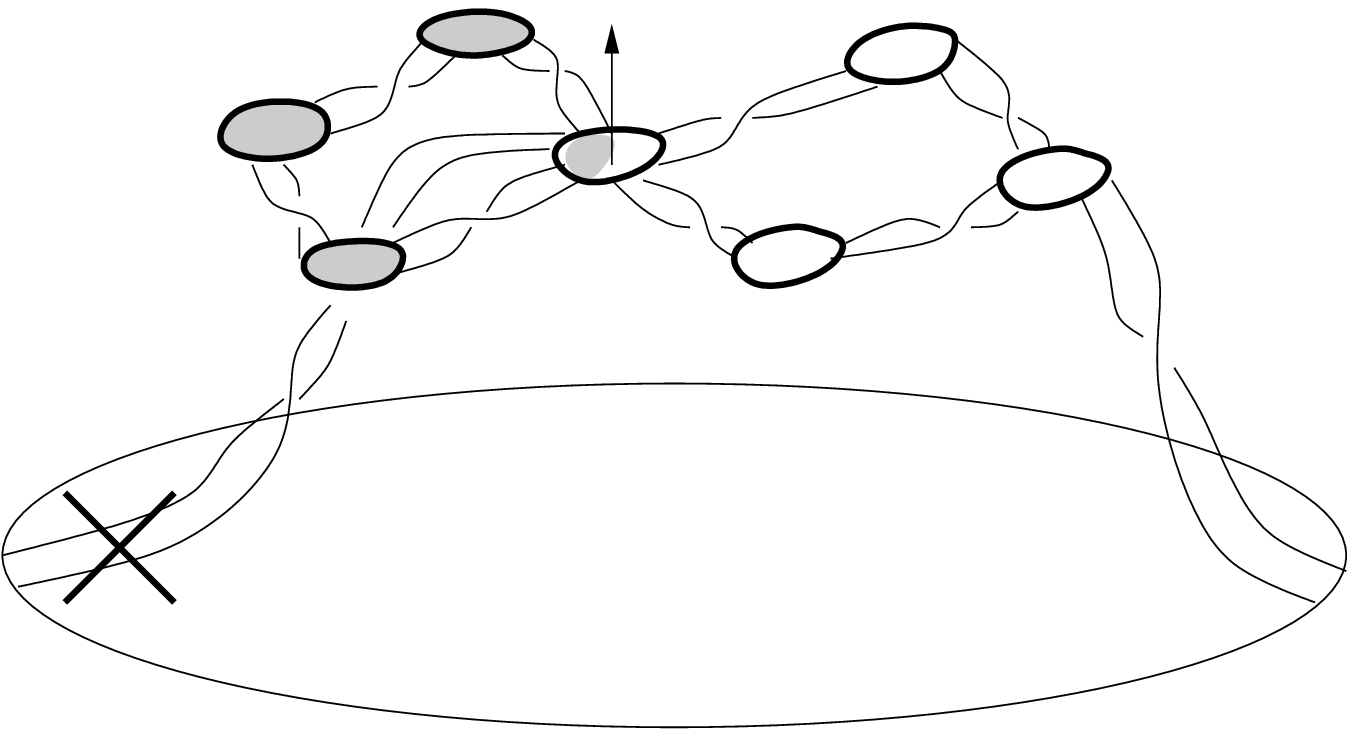}{.250}}
\def\BIIDeggBiIIDpan{\pic{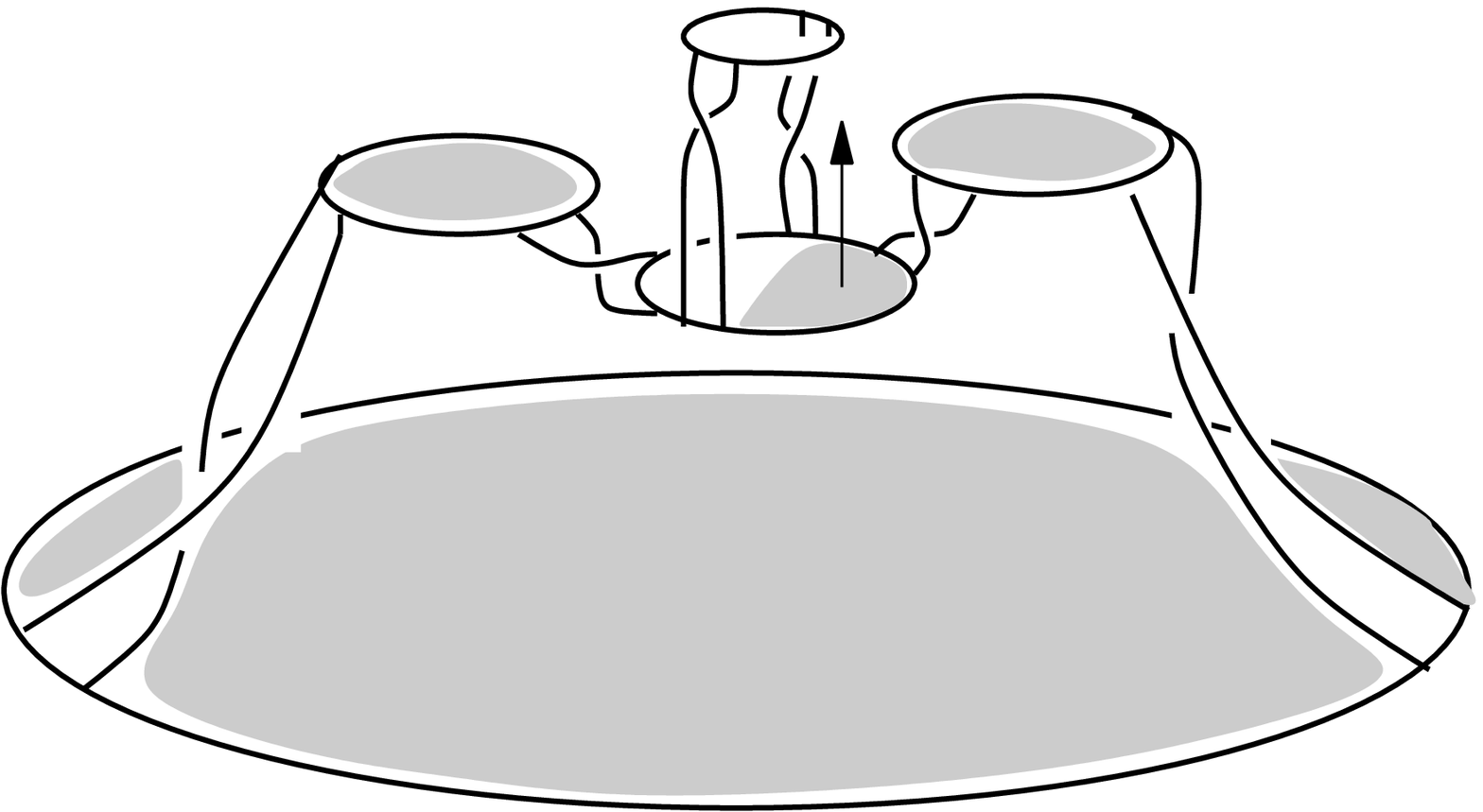}{.200}}
\def\BIIDpanBiIIDpan{\pic{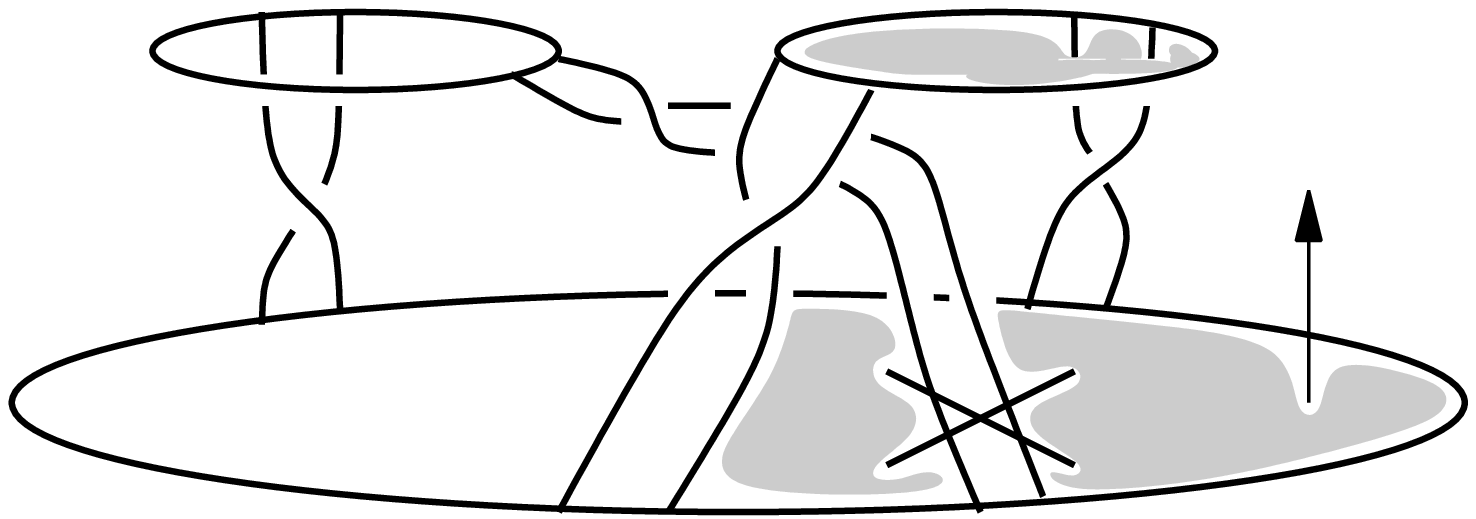}{.200}}
\def\ExampleLink{\pic{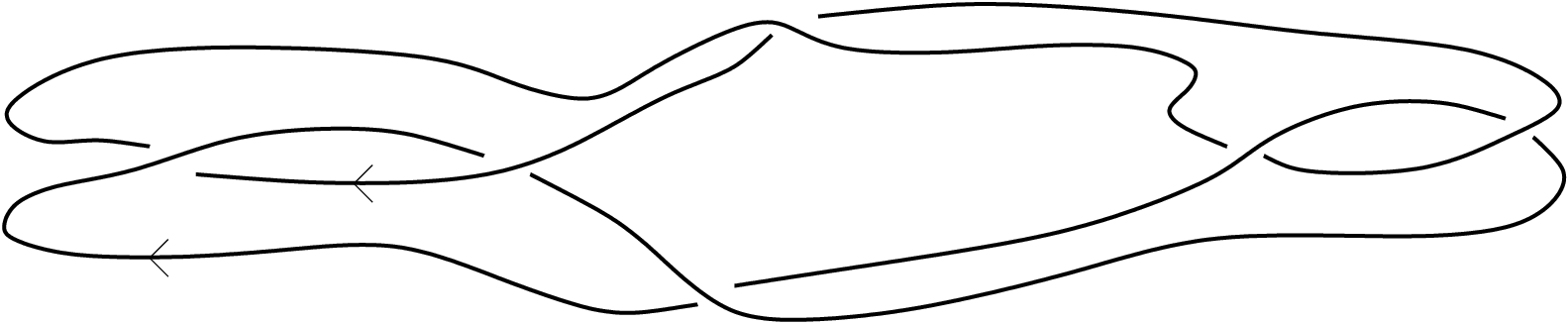}{.200}} 
\def\ExampleSeifertCircles{\pic{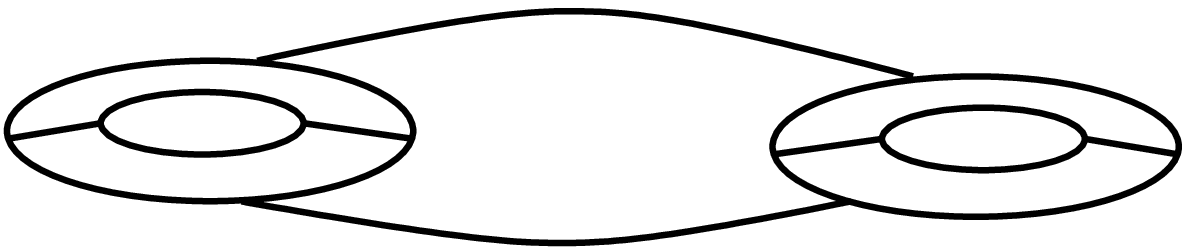}{.450}} 
\def\ExampleSeifertGraph{\pic{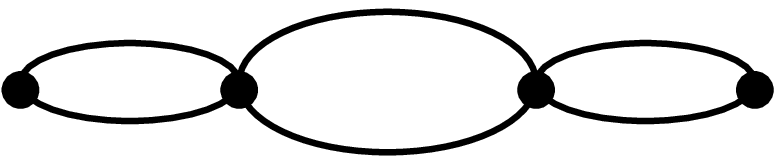}{.450}} 
\def\ExampleBlocksSeifertGraph{\pic{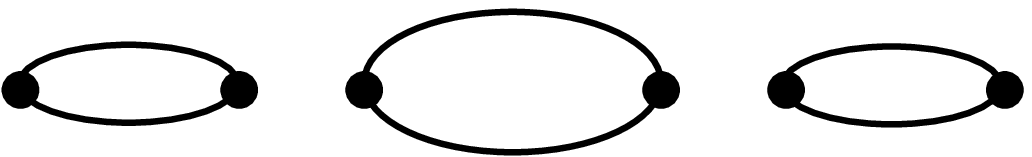}{.450}} 
\def\ExampleProjectionSurface{\pic{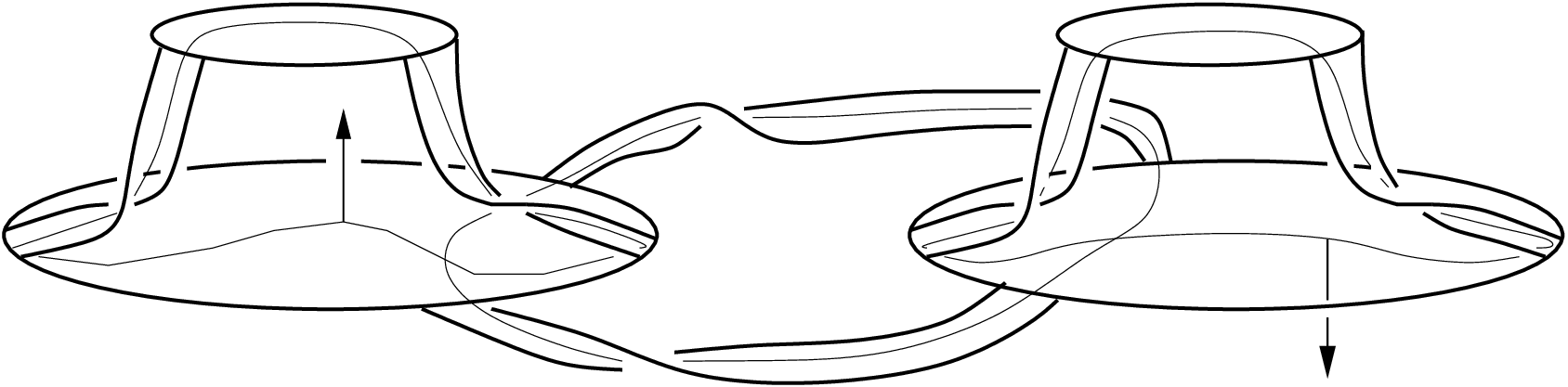}{.300}} 
\def\Blockgraph{\pic{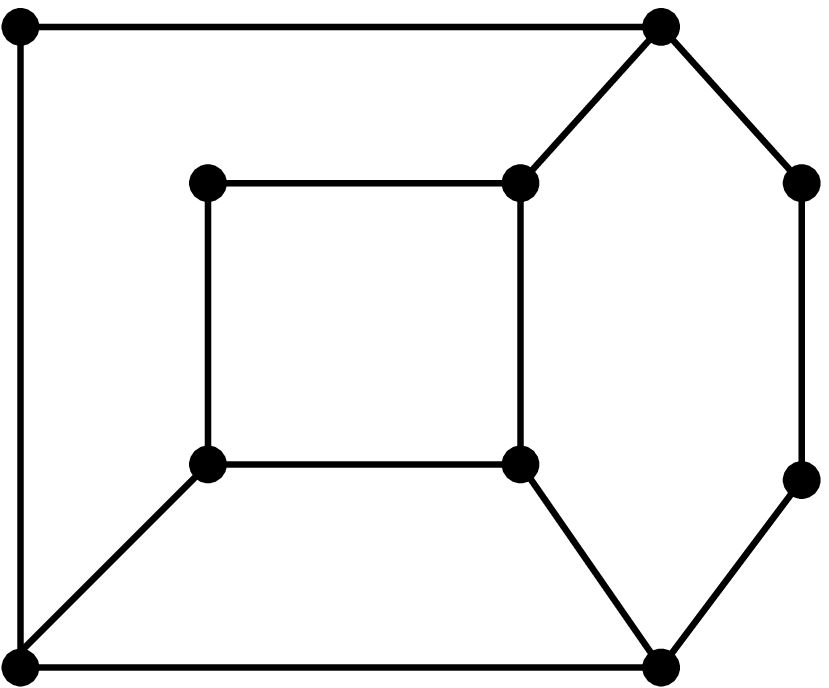}{.250}}
\def\enfrentados{\pic{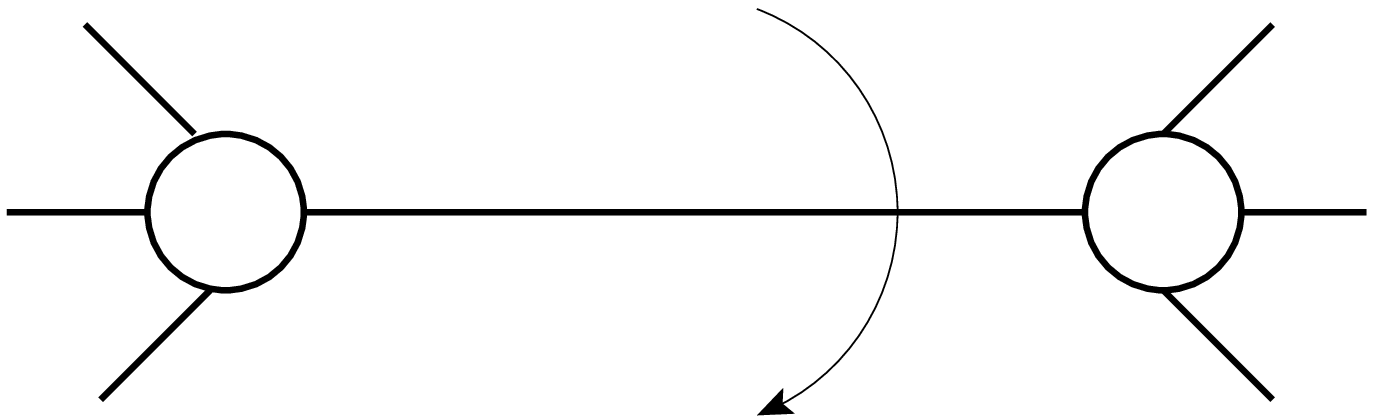}{.300}}
\def\dualchico{\pic{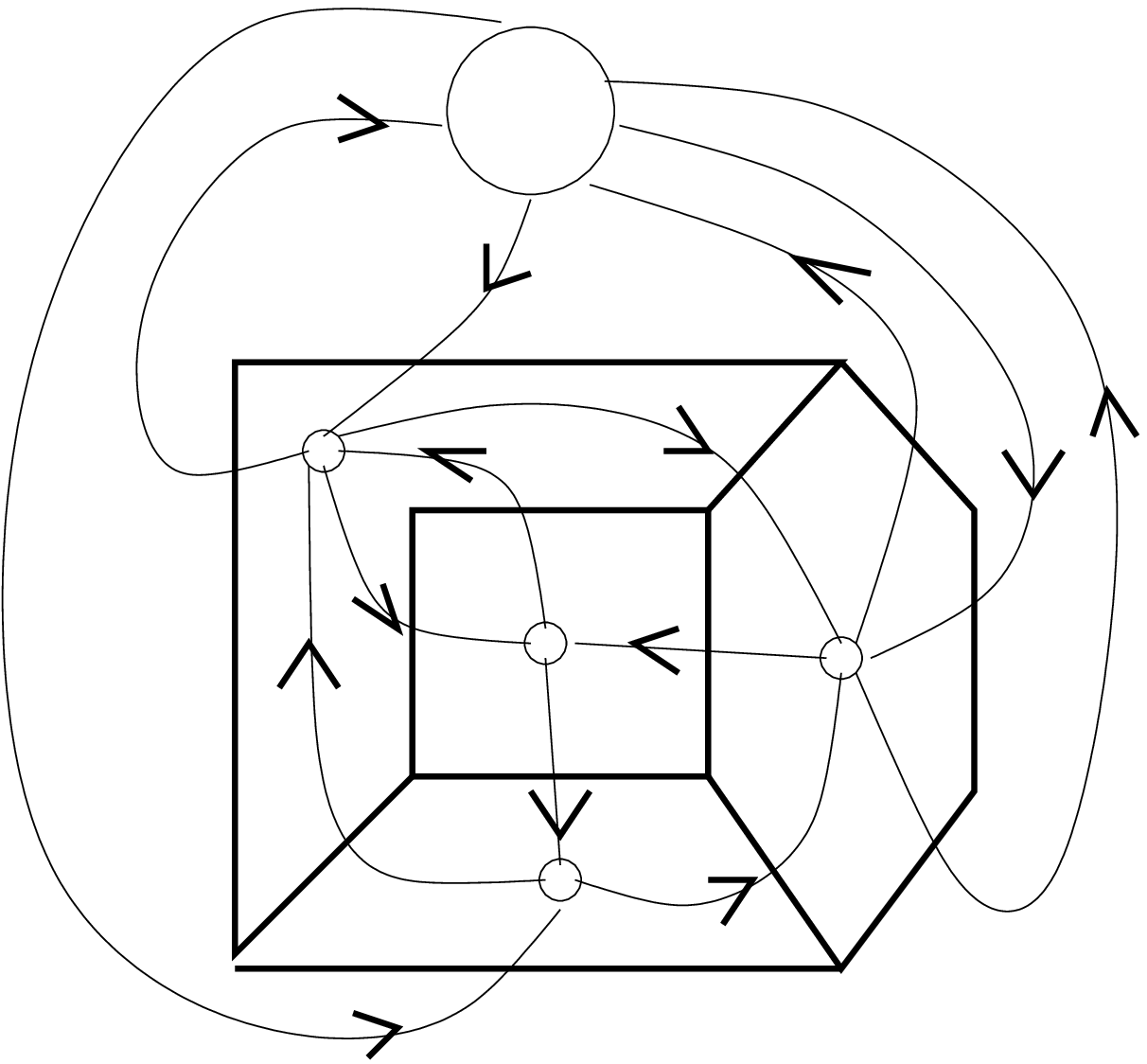}{.380}}
\def\consecutiveedges{\pic{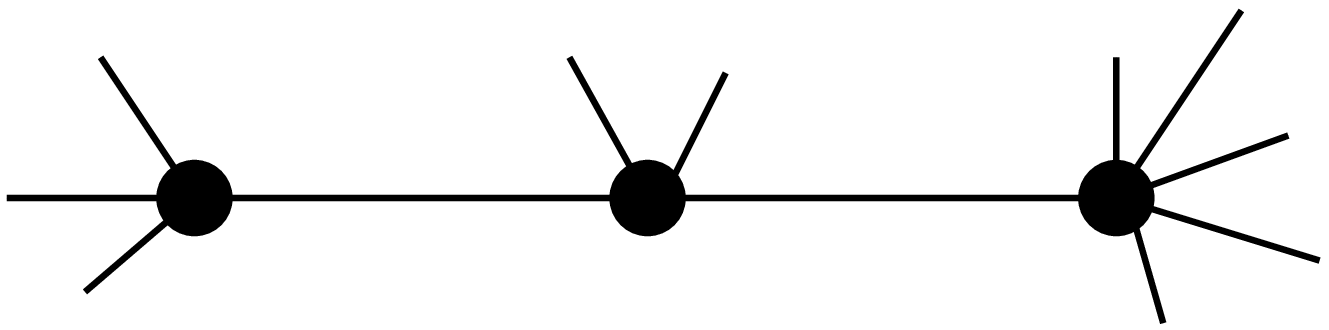}{.400}}
\def\GenusTwoHomogeneousGraphOneBlock{\pic{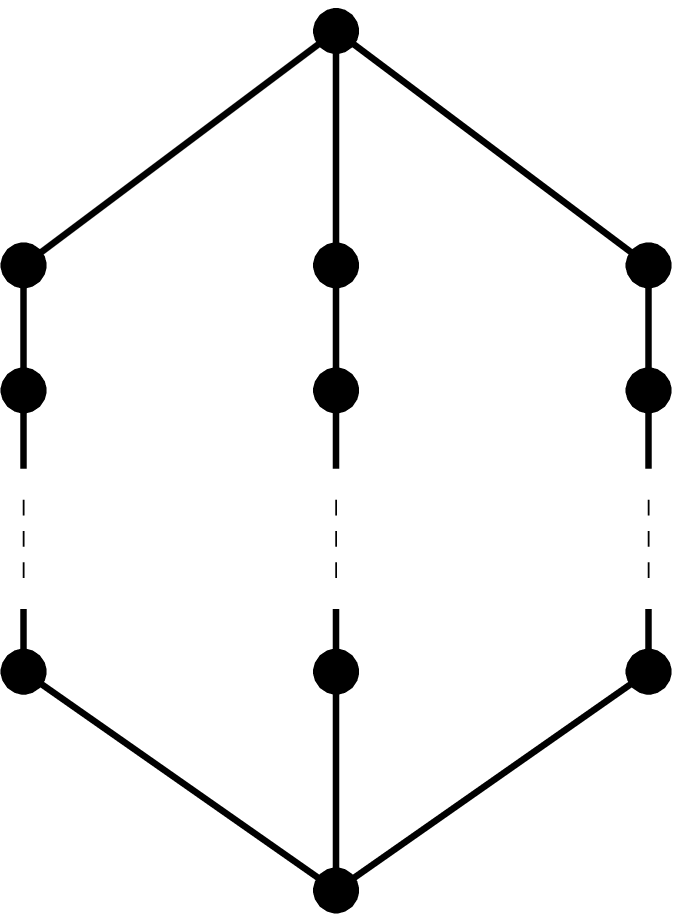}{.250}}
\def\GenusTwoHomogeneousGraphTwoBlocks{\pic{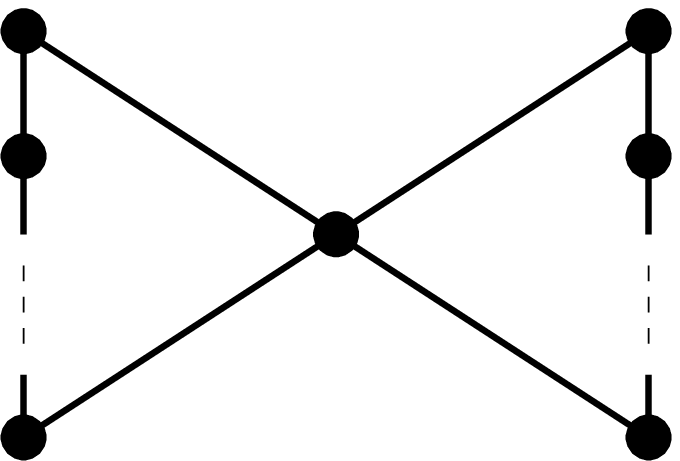}{.250}}
\def\ConfigurationsSeifertGraphsTwoBlocks{\pic{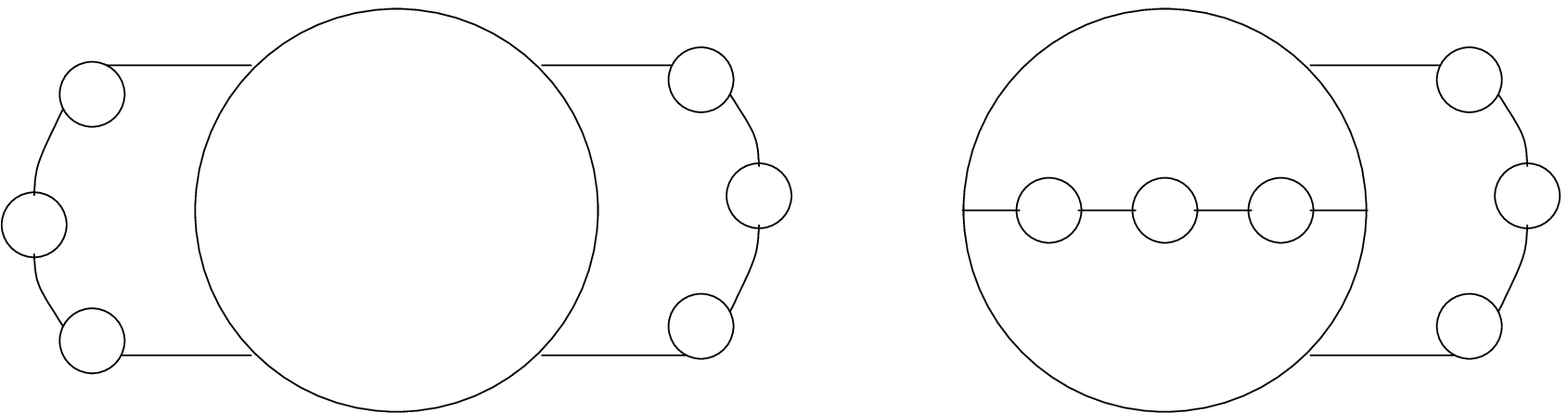}{.200}}
\def\OneBlockNoKnot{\pic{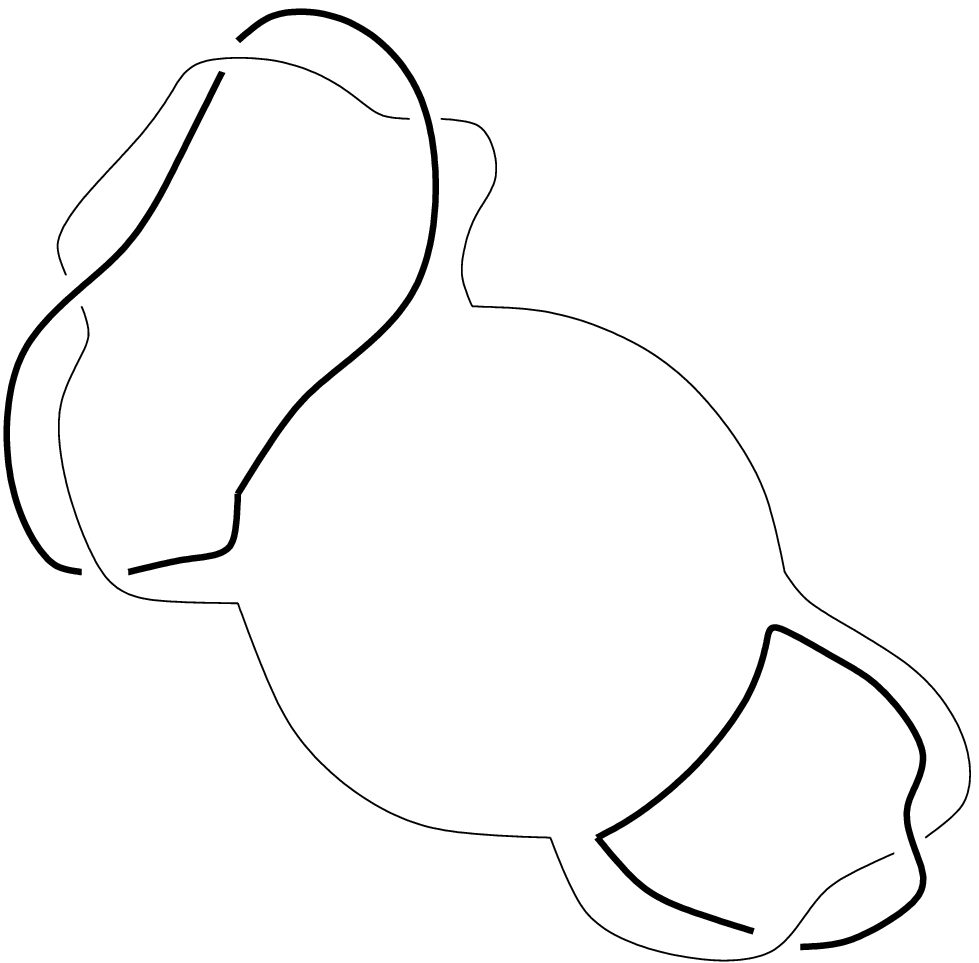}{.200}}
\def\OneBlockYesKnot{\pic{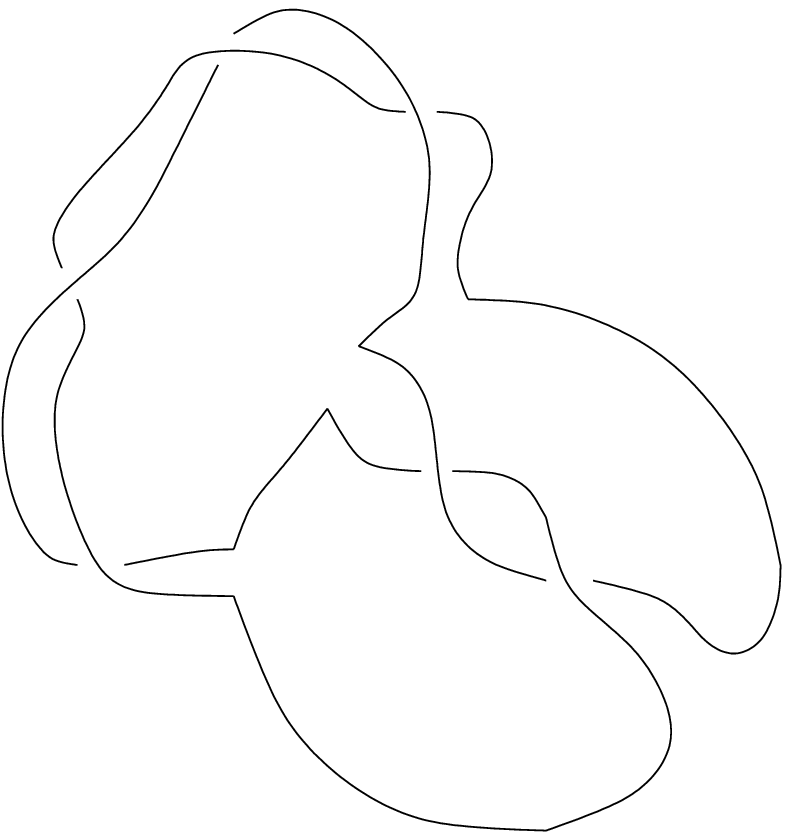}{.200}}
\def\ClaimExtra{\pic{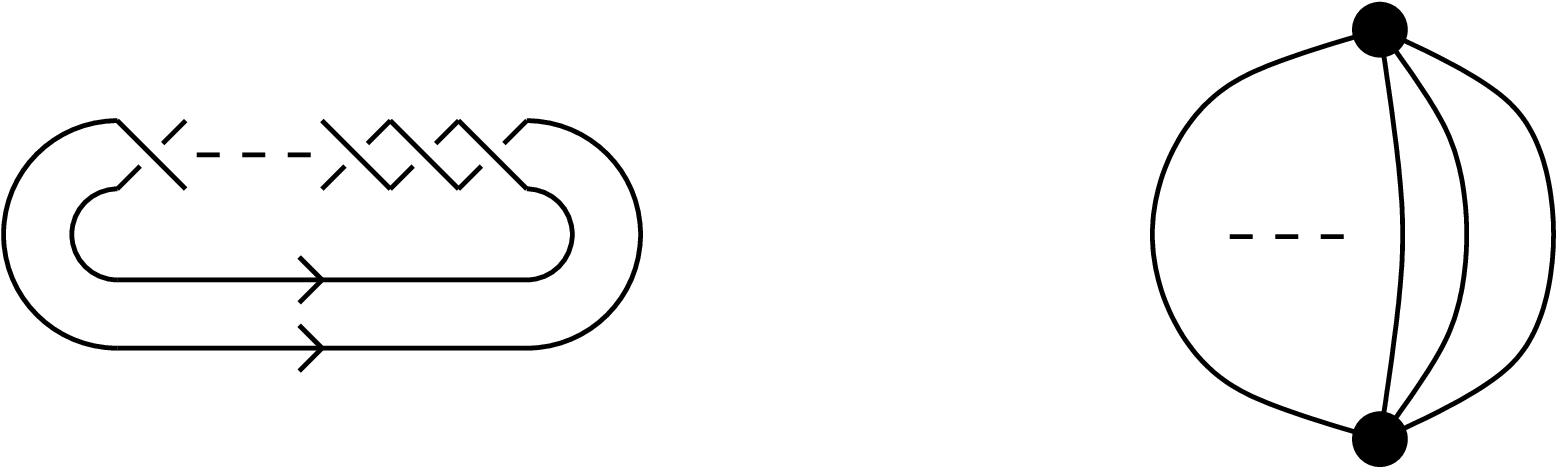}{.200}}
\def\SpecialAlternatingDiagram{\pic{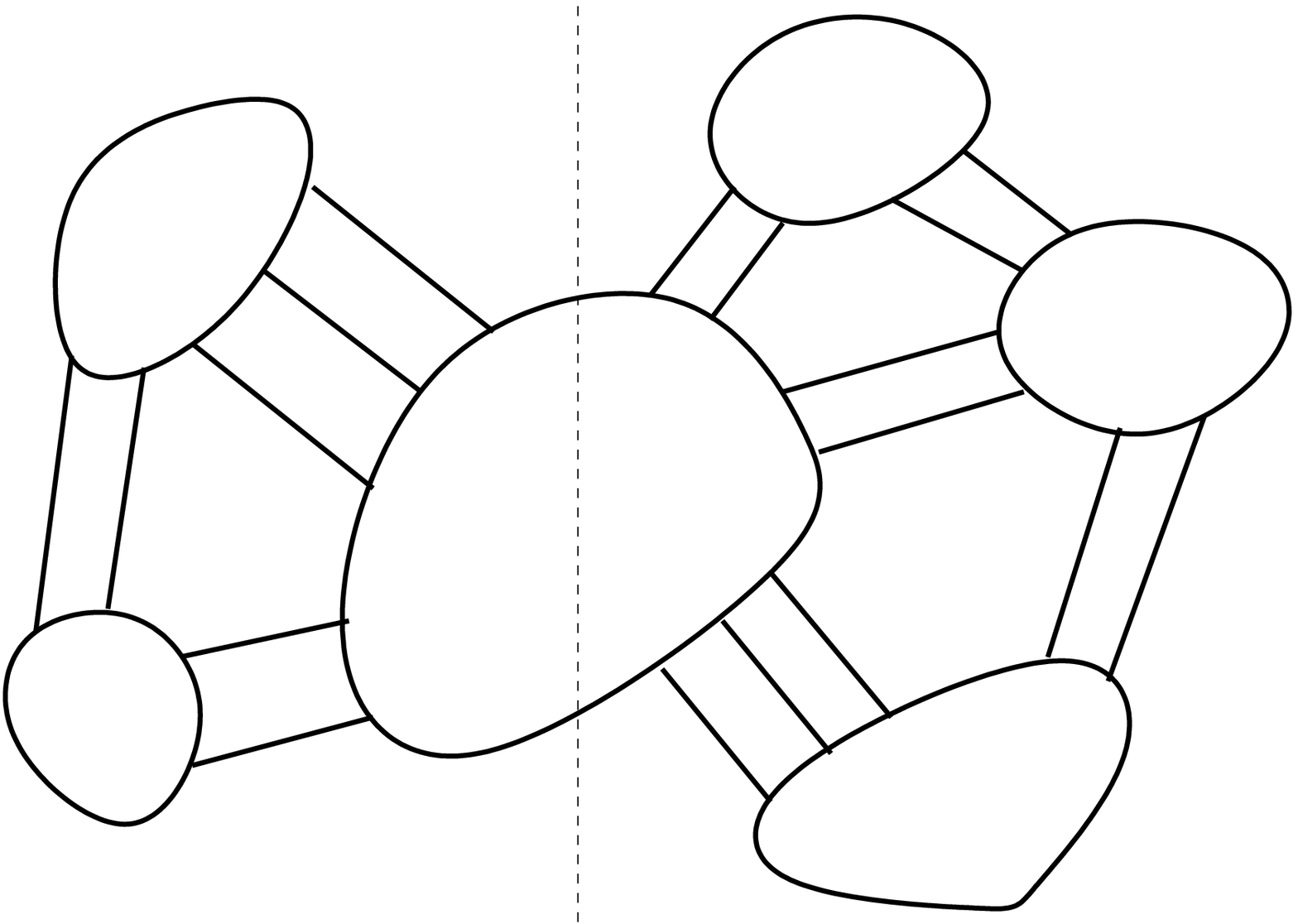}{.200}}
\def\FormOfB{\pic{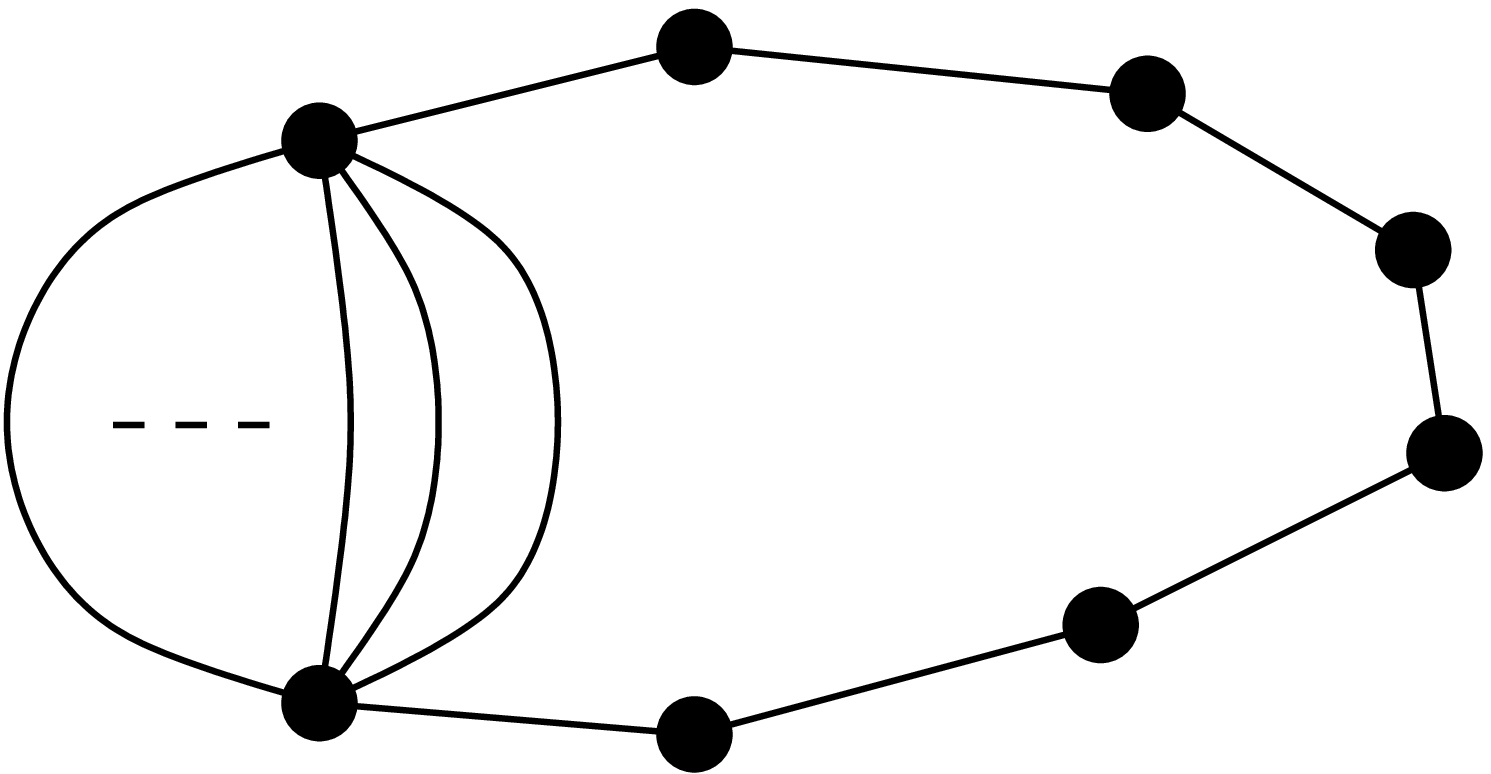}{.200}}

\newtheorem{theorem}{Theorem}

\newtheorem{corollary}[theorem]{Corollary}
\newtheorem{lemma}[theorem]{Lemma}
\newtheorem{remark}{Remark}
\newtheorem{example}{Example}

\renewenvironment{proof}[1][Proof]{\textit{#1.} }
{\hfill \rule{0.5em}{0.5em}}

\newcommand{\Z}{Z\!\!\!Z}

\begin{document}

\title{Homogeneous links and the Seifert matrix \footnote{To appear in Pacific Journal of Mathematics}}
\author{P. M. G. Manch\'on}
\maketitle

\vspace{-0.5cm}

\begin{abstract}
Homogeneous links were introduced by Peter Cromwell, who proved that the projection surface of these links, that given by the Seifert algorithm, has minimal genus. Here we provide a different proof, with a geometric rather than combinatorial flavor. To do this, we first show a direct relation between the Seifert matrix and the decomposition into blocks of the Seifert graph. Precisely, we prove that the Seifert matrix can be arranged in a block triangular form, with small boxes in the diagonal corresponding to the blocks of the Seifert graph. Then we prove that the boxes in the diagonal has non-zero determinant, by looking at an explicit matrix of degrees given by the planar structure of the Seifert graph. The paper contains also a complete classification of the homogeneous knots of genus one.
\end{abstract}

\section{Introduction} \label{introS}
Throughout this paper, we assume that all links and diagrams are oriented. Let~$F$ be a spanning surface for an oriented link $L$, and let $b:F \times [0,1] \rightarrow \mathbb{R} ^3$ be a regular neighbourhood. Identify $F$ with $F \times \{0\}$. The associated Seifert matrix $M=(a_{ij})_{1\leq i,j \leq n}$ with order $n$ is defined by the linking number $a_{ij}={\rm lk}(a_i,a_j^+)$ where the $a_i$'s are simple closed oriented curves in $F$ whose homology classes form a basis ${\cal B}$ of $H_1(F)$, and $a_i^+=b(a_i \times 1)$ is the lifting of $a_i$ out of $F$, in $F\times \{ 1\}$. Then $n={\rm rk }(H_1(F))=2g(F)+\mu -1=1-\chi(F)$ where $g(F)$ and $\chi(F)$ are the genus and the Euler characteristic of $F$ respectively, and $\mu$ is the number of components of the link. Homology with coefficients in $\Z$ is assumed along this paper.

\vspace{0.05cm}

Let $\nabla _L(z)$ and $\Delta _L(x)$ be the Conway and Alexander polynomials of $L$, in variables $z$ and $x$ respectively, as defined in \cite{Peter}. Since $\nabla _L(z)=\Delta _L(x)={\rm det}(xM-x^{-1}M^t)$ after the substitution $z=x^{-1}-x$, we have that the coefficient $c$ of the highest degree term in $\nabla _L(z)$ is equal to $(-1)^n{\rm det} (M)$ and the degree of $\nabla_L(z)$ is $n$, whenever ${\rm det}(M)$ does not vanish. In general ${\rm deg}(\nabla_L(z)) \leq n$, which provides the famous lower bound of the genus ${\rm deg}(\nabla _L(z))-\mu +1 \leq 2g(F)$ and in particular it allows to deduce that $F$ is a minimal genus spanning surface for $L$ if ${\rm det}(M)$ does not vanish.

\vspace{0.05cm}

Suppose now that the spanning surface $F$ has been constructed by applying the Seifert algorithm to a diagram $D$ of the link $L$. We briefly summarize the main features of this construction: start with a diagram $D$ in the $xy$-plane. For each Seifert circle $\alpha$ a Seifert disc $a$ is built in the plane $z=k$, if there are exactly $k$ Seifert circles that contain $\alpha$; we say that the height of $a$ is $k$ and write $h(a)=k$. This collection of discs lives in the upper half-space $\mathbb{R} ^3_+$ and they are stacked in such a way that when viewed from above, the boundary of each disc is visible. To complete the projection surface, insert small twisted rectangles (called bands from now on) at the site of each crossing, choosing the half-twist according to the corresponding crossing. Following Cromwell \cite{Peter}, we call $F$ a projection surface. 

\vspace{0.05cm}

We can now define a graph $G$ contained in $F$ as follows: take a vertex in each Seifert disc of $F$ and, if two discs are joined by a band, join the corresponding vertices by an edge contained in the band. In addition, we label the edge with the sign of the associated crossing in the diagram $D$. This graph, called the Seifert graph of $D$, is in fact a planar graph. The rank ${\rm rk }(G)$ of $G$, as defined in graph theory, is one minus the number of vertices plus the number of edges. Since $\chi(F)=s(D)-c(D)$ where $s(D)$ is the number of Seifert circles and $c(D)$ is the number of crossings of $D$, it follows that ${\rm rk }(G)={\rm rk }(H_1(F))$.

\vspace{0.05cm}

In general, we can consider the decomposition $G=B_1 \cup \dots \cup B_k$ of the graph $G$ into its blocks, which are the maximal connected subgraphs without cut vertices. The part of the projection surface (bands and Seifert discs) that corresponds to a block $B_i$ is a submanifold of $F$ and will be denoted by $F_{B_i}$, or simply $F_i$. The graph $G$ is a deformation retract of the surface $F$, taking $F_i$ onto $B_i$; in particular $H_1(F) \cong H_1(G)$ taking $H_1(F_i)$ onto $H_1(B_i)$ and ${\rm rk }(G)={\rm rk }(H_1(G))$, an equality which is sometimes taken as a definition. Now, a basis of $H_1(G)$, hence a basis ${\cal B}$ of $H_1(F)$, can be obtained by juxtaposing basis ${\cal B}_i$ of $H_1(B_i)$, since the cycles in $G$ are precisely the cycles of its blocks (\cite{Diestel}, lemma 3.1.1). In particular, the rank of $G$ is the sum of the ranks of its blocks.

\vspace{0.05cm}

Let $M_i$ be the Seifert matrix defined by any basis ${\cal B}_i$ of $H_1(B_i)$ (hence of $H_1(F_i)$), $i=1, \dots , k$. Our main result is the following:

\vspace{0.05cm}

\noindent {\bf Theorem 4.} {\it Let $D$ be a connected diagram of an oriented link $L$, $G$ the corresponding Seifert graph and $G=B_1 \cup \dots \cup B_k$ the decomposition of $G$ into~blocks. Then there is an order in the set of blocks of $G$ for which the Seifert matrix for the projection surface is upper block triangular. More precisely, if $M_i$ is the Seifert matrix that corresponds to any basis ${\cal B}_i$ of $H_1(B_i)$, $i = 1, \dots , k$, there exists a permutation $\sigma \in S_k$ such that the Seifert matrix adopts the following form:
$$\bordermatrix{
                \ & {\cal B}_{\sigma (1)}^+ & {\cal B}_{\sigma (2)}^+ & \dots  & {\cal B}_{\sigma (k)}^+ \cr 
   {\cal B}_{\sigma (1)} & M_{\sigma (1)}   & 0                & \dots  & 0                \cr
   {\cal B}_{\sigma (2)} & *                & M_{\sigma (2)}   & \ddots & \vdots           \cr
   \ \ \vdots         & \vdots           & \ddots           & \ddots & 0                \cr
   {\cal B}_{\sigma (k)} & *                & \dots            & *      & M_{\sigma (k)}
}$$
}
\vspace{0.05cm}

A link is homogeneous if it has a homogeneous diagram, which is a diagram in which all the edges of each block of its Seifert graph have the same sign. Alternating and positive diagrams (links) are homogeneous diagrams (links). The knot $9_{43}$ is an example of homogeneous link which is neither positive nor alternating. Homogeneous links were introduced in 1989 by Cromwell \cite{PeterPaper}. In Knot Theory the adjective homogeneous was first applied to a certain class of braids by  Stallings \cite{Stallings}. Certainly, the closure of a homogeneous braid is a homogeneous diagram, although there are homogeneous links which cannot be presented as the closure of a homogeneous braid, just as there are alternating links which cannot be presented as the closure of alternating braids. In \cite{PeterPaper} (see also \cite{Peter}) Cromwell proved the following basic result on homogeneous links:

\vspace{0.1cm}

\noindent {\bf Theorem (Cromwell).} {\it Let $D$ be a connected homogeneous diagram of an oriented homogeneous link $L$ and $G$ the corresponding Seifert graph. Then the highest degree of $\nabla _L(z)$ is the rank of $G$. Furthermore, let $G=B_1 \cup \dots \cup B_k$ be the decomposition of $G$ into blocks and $M_i$ the corresponding Seifert matrices, $i=1, \dots , k$. Then ${\rm det}(M_i)\not= 0$ for all $i=1,\dots , k$ and the leading coefficient of $\nabla _L(z)$ is 
$$
\prod_{i=1}^k \epsilon _i ^{r_i} |{\rm det}(M_i)|
$$
where $\epsilon_i$ is the sign of the edges in $B_i$ and $r_i={\rm rk }(B_i)$.}

\vspace{0.1cm}

\noindent {\bf Corollary.} {\it A projection surface constructed from a connected homogeneous diagram of an oriented link is a minimal genus spanning surface for the link.}

\vspace{0.1cm}

Cromwell's proof is based on a previous construction of a specific resolving tree for calculating the Conway polynomial (\cite{Peter}, Lemma 7.5.1). This means that no crossing is switched more than once on any path from the root of the tree to one of its leaves. The skein relation is then considered, at both the level of the diagram and the corresponding Seifert graph, having in mind that to obtain terms involving powers of $z$ when resolving the resolution tree, a crossing must be smoothed in the diagram $D$, or equivalently, an edge must be deleted from the graph $G$. A direct proof of the corollary has been recent and independently suggested by M. Hirasawa. The proof, outlined in a paper by Tetsuya Abe \cite{Abe} (see also \cite{Ozawa}), is strongly based on a difficult result by Gabai \cite{Gabai2}, which states that the sum of Murasugi of minimal genus surfaces is a minimal genus surface. Hirasawa applies this result to the portions $F_i$ above defined. 

\vspace{0.1cm}

In this paper we give a different proof of Cromwell's theorem, based on the close relation between the Seifert matrix and the decomposition into blocks of the Seifert graph stated in Theorem~\ref{DiagonalMatrixTheorem}. The key point is the understanding of how the parts of the projection surface corresponding to the blocks are geometrically positioned among them. We remark that Theorem~\ref{DiagonalMatrixTheorem} can be useful even in the case in which the diagram is not homogeneous. A special case was already considered by Melvin and Morton in their work on fibred knots of genus two formed by plumbing Hopf bands \cite{Morton}. We deal with this topic in Section~\ref{triangularS}. 

\vspace{0.05cm}

Since a homogeneous block of the Seifert graph corresponds to an alternating diagram, it follows that each little box in the diagonal of the Seifert matrix has non-zero determinant, according to the work by K. Murasugi \cite{MurasugiI&II} (see also \cite{Murasugi}) and independently Crowell \cite{Crowell}. Murasugi's proof was accomplished by working on the Alexander matrix of the Dehn presentation of the link, while Crowell worked with the Wirtinger presentation of the fundamental group. In this paper we will prove this result, second ingredient of our argument, by looking at an explicit matrix of degrees which uses the planar structure of the Seifert graph (Theorem~\ref{MatrixPropertiesTheorem}). This will be treated in Section~\ref{boxS}.

\vspace{0.1cm}

Finally, Section~\ref{ClassificationS} contains a complete classification of genus one homogeneous knots.

\section{An order for the blocks and the Seifert matrix}\label{triangularS}
The main achievement of this paper is to prove that there is a certain ordered basis of the first homology group of the projection surface for which the Seifert matrix has a block triangular form. We need first to prove that, in a certain sense, there are only two types of blocks, or more precisely, there are only two possible configurations for the portions $F_B$ associated to a block $B$.

Let $a,b$ be two Seifert discs. We say that $a$ contains $b$ (write $a \supset b$) if the projection onto the $xy$-plane of $a$ contains that of $b$. Equivalently, the Seifert circle associated to $a$ contains that associated to $b$, in the $xy$-plane.
\begin{figure}[ht!]
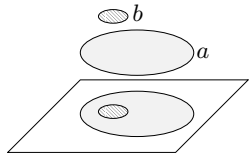

\labellist
\small
\pinlabel {$a$} at 390 200
\pinlabel {$b$} at 260 280
\endlabellist
  \begin{center}
\nestedcircles
  \end{center}
\caption{Nested circles: $a$ contains $b$}\label{nestedcirclesF}
\end{figure}

\begin{remark}\label{Nota}
If we project the projection surface onto the $xy$-plane, the only self-intersections of its boundary are given by the crossings of the original diagram $D$, and they are produced by the half-twists of the bands.
\end{remark}

In particular the arrangement in Figure \ref{NoPossibleF} is not possible, and the following result follows:
\begin{lemma}\label{lema}
Let $a,b$ be two Seifert discs connected by a band. Then exactly one of the following three statements holds:
\begin{enumerate}
    \item $a \supset b$ and $h(b)=h(a)+1$.
    \item $b \supset a$ and $h(a)=h(b)+1$.
    \item $h(a)=h(b)$.
\end{enumerate}
\end{lemma}
\begin{figure}[ht!]
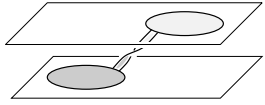

  \begin{center}
\NoPossible
  \end{center}
\caption{This arrangement of the Seifert discs is not possible}\label{NoPossibleF}
\end{figure}

The proof is easy and left to the reader. Now, we can prove that there are basically two types of blocks. Precisely:

\begin{theorem}\label{TwoTypesTheorem}
Let $D$ be a diagram, $F$ its projection surface and $G$ the corresponding Seifert graph. Then all the Seifert discs associated to a block of $G$ have the same height, except possibly one of them which contains all the other, being its height one less.
\end{theorem}
\begin{proof}
Suppose that $a$ and $b$ are two Seifert discs with different height connected by a band, both associated to the same block. By Lemma \ref{lema} we may assume that one contains the other, let say $a \supset b$. It turns out that there is no other Seifert disc associated to the block with height lower than $b$, since that would make the vertex corresponding to $a$ a cut vertex, according to Remark~\ref{Nota}. Analogously, any other disc above $b$ would make (the vertex corresponding to) $b$ a cut vertex.
\end{proof}

Hence we have two possible arrangements for (the Seifert discs that correspond to) a block: type I ({\it fried eggs} type) and type II ({\it fried eggs with a pan} type). In a type II block, the pan is the Seifert disc with lowest height. The two types of blocks are shown in Figure \ref{typesF}.
\begin{figure}[ht!]
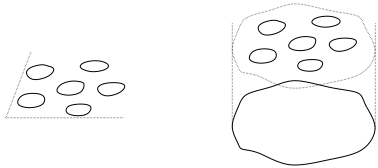

  \begin{center}
\friedeggs \qquad \qquad \friedeggswithpan
  \end{center}
\caption{Fried eggs and Fried eggs with a pan: the two types of blocks}\label{typesF}
\end{figure}

Following Cromwell \cite{PeterPaper} or Murasugi \cite{MurasugiI&II} we say that a (Seifert) circle is of type I if it does not contain any other circle; otherwise it is of type II. When a type II circle has other circles outside, it is called a decomposing circle. By definition, a special diagram does not contain any decomposing circle. Note that a type II circle is the boundary of the {\it pan} of a type II block, assuming that the diagram is connected.

\vspace{0.1cm}

Now, recall from the introduction that the part of the projection surface that corresponds to a block $B_i$ is denoted by $F_i$, which is a submanifold of $F$. Recall also that, since the cycles of a graph are the cycles of its blocks, we have that a basis of $H_1(F)$ can be obtained by juxtaposing a basis for each block.

\begin{remark}
Two different $F_i$'s can have at most one common Seifert disc, hence $F$ is the Murasugi sum of the portions $F_i$'s. The proof by Hirasawa mentioned in the introduction follows from this fact.
\end{remark}

In order to prove the main theorem, we need the following result of graph theory:

\begin{lemma} \label{OneCutVertexBlock}
Let $G$ be a connected finite graph with at least one cut vertex. Then there is a block of $G$ which has exactly one cut vertex of $G$.
\end{lemma}

\begin{proof}
It can be deduced from Proposition 3.1.2 of \cite{Diestel}. It follows a direct argument: delete any cut vertex $v_0$ of $C_0=G$ and consider $C_1=C_1'\cup \{ v_0\}$ where $C_1'$ is any connected component of $C_0-\{ v_0\}$. We remark that, under these assumptions, the cut vertices of $C_1$ are exactly the cut vertices of $G$ that lie in $C_1$, except for $v_0$, and that any block of $C_1$ is a block of $G$. If $C_1$ has no cut vertices, then it is the wanted block. Otherwise we select a cut vertex $v_1$ of $C_1$ and consider $C_2=C_2' \cup \{v_1\}$ where $C_2'$ is a connected component of $C_1-\{ v_1\}$ with $v_0 \notin C_2'$. Repeating this process, we finally get a $k \in \mathbb{N}$ such that $C_k$ has no cut vertices, hence being the wanted block. Otherwise we would obtain an infinite sequence of distinct vertices $\{ v_0, v_1, v_2, \dots \}$ in the finite graph~$G$, a contradiction. 
\end{proof}

\begin{theorem}\label{DiagonalMatrixTheorem}
Let $D$ be a connected diagram of an oriented link $L$, $G$ the corresponding Seifert graph and $G=B_1 \cup \dots \cup B_k$ the decomposition of $G$ into blocks. Then there is an order in the set of blocks of $G$ for which the Seifert matrix for the projection surface is upper block triangular. More precisely, if $M_i$ is the Seifert matrix that corresponds to any basis ${\cal B}_i$ of $H_1(B_i)$, $i=1, \dots , k$, there exists a permutation $\sigma \in S_k$ such that the Seifert matrix adopts the following form:
$$\bordermatrix{
                \ & {\cal B}_{\sigma (1)}^+ & {\cal B}_{\sigma (2)}^+ & \dots  & {\cal B}_{\sigma (k)}^+ \cr 
   {\cal B}_{\sigma (1)} & M_{\sigma (1)}   & 0                & \dots  & 0                \cr
   {\cal B}_{\sigma (2)} & *                & M_{\sigma (2)}   & \ddots & \vdots           \cr
   \ \ \vdots         & \vdots           & \ddots           & \ddots & 0                \cr
   {\cal B}_{\sigma (k)} & *                & \dots            & *      & M_{\sigma (k)}
}$$

\end{theorem}

\begin{proof}
By Lemma~\ref{OneCutVertexBlock}, there exists a block $B$ which has exactly one cut vertex. Let $D$ be the Seifert disc associated to the unique cut vertex in $B$. Translated to the surface, this means that the geometric block $F_B$ is separated from the rest of the surface $F$, with $D$ as the unique intersection.

We may assume, by induction on the number of blocks, that the Seifert matrix for $F \backslash (F_B \backslash D)$ is upper triangular for a suitable order of the rest of blocks $B_i$'s. Suppose now that the positive orientation of the disc $D$, that looking at $F \times \{ 1\}$, is upwards. Then, the basis that corresponds to the block $B$ must be added
\begin{itemize}
  \item at the beginning if $B$ is of type I, or $D$ is an egg of the type II block $B$,
  \item at the end if $D$ is the pan of the type II block $B$.
\end{itemize}
Indeed, if the positive orientation of $D$ were downward, these statements must be interchanged.

In the following displayed figures, the shadowed discs correspond all to the block $B$; the disc $D$, partially shadowed, is part of the two considered blocks, $B$ and any other block $B_i$ previously ordered. On $D$ there is an oriented arrow looking upwards, indicating the positive orientation. Suppose now that $g, g_i \in H_1(F)$ correspond to the blocks $B$ and $B_i$ respectively. We have to analyse the three possible cases:
\begin{enumerate}
  \item Suppose that $B$ is of type I. We have to see that ${\rm lk}(g,g_i^+)=0$. This can be easily checked if $B_i$ is of type I, or $B_i$ is of type II and the disc $D$ is its pan. And it is also true if $B_i$ is of type II being $D$ an egg of $B_i$, since in this case the eggs would be on different half parts of the pan. To see this, project both blocks 
$B$ and $B_i$ onto the plane $z=h(D)-1$, hence the Seifert discs at height $h(D)$ are now nested inside the pan of the block $B_i$ ({\it all the eggs in the same pan}). By remark~\ref{Nota} there is no intersections other than those given by the half-twists of the bands, which means that the two blocks are basically in separated half parts of the pan of 
$B_i$. In particular, a band like the crossed one in Figure~\ref{BIBiIIDeggF} is not possible.
\begin{figure}[ht!]
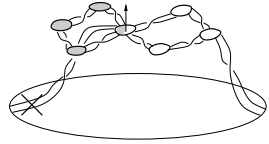

\begin{center}
\BIBiIIDegg
\end{center}
\caption{This situation is not possible}\label{BIBiIIDeggF}
\end{figure}

  \item Suppose that $B$ is of type II, and the disc $D$ is an egg of $B$. As in the previous case, we have to see that ${\rm lk}(g,g_i^+)=0$. This can be easily checked if any other block $B_i$ is of type I, or (see Figure~\ref{BIIDeggBiIIDpanF}) $B_i$ is of type~{II} being the disc $D$ the pan of $B_i$.
\begin{figure}[ht!]
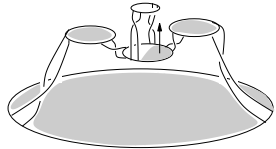

  \begin{center}
  \BIIDeggBiIIDpan
  \end{center}
\caption{${\rm lk }(g,g_i^+)=0$, $g$ defined by $B$, $g_i$ defined by $B_i$}\label{BIIDeggBiIIDpanF}
\end{figure}

Note that $D$ cannot be egg of another type II block $B_i$. Indeed, in that case and again by remark~\ref{Nota}, the pan should be the same for $B$ and $B_i$, hence the blocks $B$ and $B_i$ would share at least two vertices. But, by their maximality, different blocks of $G$ overlap in at most one vertex.

  \item Suppose that $B$ is of type II, and the disc $D$ is its pan. In this case we have to see that ${\rm lk}(g_i,g^+)=0$. This can be easily checked if the block $B_i$ is of type I, or the disc $D$ is an egg of a type II block $B_i$. And it is also true if the disc $D$ is the pan of another type II block $B_i$, since in this case, by a similar argument to that used in the first case, the eggs would be on different half parts of the pan, the crossed band in Figure~\ref{BIIDpanBiIIDpanF} being not possible. 
\begin{figure}[ht!]
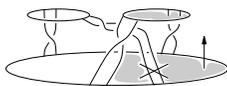

  \begin{center}
  \BIIDpanBiIIDpan
  \end{center}
\caption{${\rm lk }(g_i,g^+)=0$, both type II blocks sharing the pan}\label{BIIDpanBiIIDpanF}
\end{figure}
\end{enumerate}
\end{proof}

\begin{example}
{\rm Suppose that we wish to find the block triangular form for the Seifert matrix of the link shown on top of Figure~\ref{ExampleF}.
\begin{figure}[ht!]
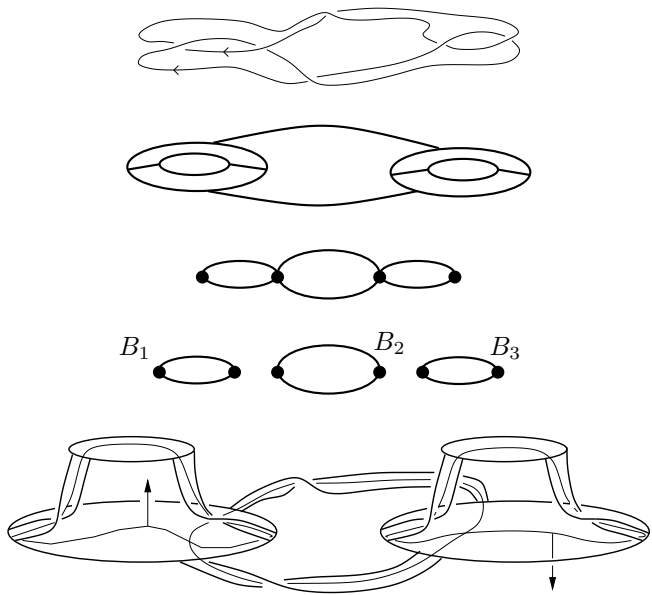

\labellist
\pinlabel {$B_1$} at -5 -520
\pinlabel {$B_2$} at 500 -510
\pinlabel {$B_3$} at 730 -520
\endlabellist

  \begin{center}
  \ExampleLink \\
  \ \\
  \ \\
  \ExampleSeifertCircles \\
  \ \\
  \ \\
  \ExampleSeifertGraph \\ 
  \ \\
  \ \\
  \ExampleBlocksSeifertGraph \\
  \ \\
  \ \\
  \ExampleProjectionSurface
  \end{center}
\caption{A $2$-component link diagram, the Seifert circles, the Seifert graph, decomposition into blocks of the Seifert graph and the projection surface}\label{ExampleF}
\end{figure}

We consider the Seifert graph, with blocks $B_1, B_2$ and $B_3$, from left to right, and the projection surface. We can consider $B=B_1$ as the block with only one cut vertex. Then, if the positive orientation of $D$ is upwards, for the other two blocks the suitable basis is given by the order of blocks $\{ B_3, B_2\}$, which gives the matrix
$$
\begin{array}{c|c|c}
& B_3^+ & B_2^+ \\
\hline
B_3 & * & 0 \\
\hline
B_2 & * & *
\end{array}
$$

Since $B=B_1$ is a block of type II and the disc $D$ that corresponds to the cut vertex is a pan of $B$, according to the proof of Theorem~\ref{DiagonalMatrixTheorem} we must add the basis for $B$ at the end, obtaining the order $\{ B_3, B_2, B_1\}$ and the matrix
$$
\begin{array}{c|c|c|c}
& B_3^+ & B_2^+ & B_1^+ \\
\hline
B_3 & * & 0 & 0 \\
\hline
B_2 & * & * & 0 \\
\hline
B_1 & * & * & *
\end{array}
$$
}
\end{example}

\section{The box matrix associated to a block}\label{boxS}
Recall from the introduction that the coefficient $c$ of the highest degree term in $\nabla _L(z)$ is equal to $(-1)^n{\rm det} (M)$ and the degree of $\nabla_L(z)$ is $n={\rm rk }(H_1(F))$, whenever ${\rm det}(M)$ does not vanish. By Theorem~\ref{DiagonalMatrixTheorem} ${\det }(M)=\prod _{i=1}^k {\det }(M_i)$ where $M_i$ is the Seifert matrix that corresponds to the surface $F_i$ associated to the block $B_i$ of $G$. Then, in order to prove the theorem stated in the introduction, it is enough to show that, if $B_i$ is a block with rank $r_i$ and all its edges have sign $\epsilon_i$, then the determinant of its Seifert matrix does not vanish and has sign $(-\epsilon_i )^{r_i}$. Indeed, since $n={\rm rk }(G)$ is the sum of the ranks $r_i$ of its blocks, we would have
\begin{eqnarray*}
c & = & (-1)^n {\det }(M)  \\
  & = & (-1)^n \prod _{i=1}^k {\det }(M_i) \\
  & = & (-1)^n \prod _{i=1}^k (-\epsilon _i)^{r_i}|\det (M_i)| \\
  & = & \prod _{i=1}^k \epsilon _i^{r_i} |{\det }(M_i)| \ \not= \ 0.
\end{eqnarray*}

Now, the part of the diagram that corresponds to a homogeneous block is alternating (in fact, it is a special alternating diagram), and the result for these links follows from the work by Murasugi \cite{Murasugi} and Crowell \cite{Crowell} fifty years ago. Murasugi's proof was accomplished by working on the Alexander matrix of the Dehn presentation, while Crowell worked with the Wirtinger presentation of the fundamental group of the link. In fact, Crowell's paper rests on a striking application of a graph theoretical result, the Bott-Mayberry matrix tree theorem, an approach which is also explained in Proposition 13.24 of \cite{Burde}. In this section we will prove it (Theorems~\ref{MatrixPropertiesTheorem} and \ref{DeterminantTheorem}) by looking at an explicit matrix of degrees defined using the planar structure of the Seifert graph.

Let $D$ be an oriented diagram, $F$ its projection surface and $G$ the corresponding Seifert graph. Let $B$ be a block of $G$. A basis $\{ g_1, \dots , g_r\}$ of $H_1(B)$ (hence of $H_1(F_B)$) can be obtained collecting the counterclockwise oriented cycles defined by the boundaries of the bounded regions $R_i$ defined by $B$. Let $R_{r+1}$ be the unbounded region defined by this planar graph $B$ (see Figure \ref{BlockgraphF}).
\begin{figure}[ht!]
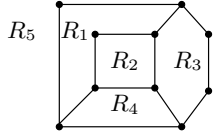

\labellist
\small
\pinlabel {$R_1$} at 30 150
\pinlabel {$R_2$} at 105 105
\pinlabel {$R_3$} at 200 105
\pinlabel {$R_4$} at 105 40
\pinlabel {$R_5$} at -50 150
\endlabellist
  \begin{center}
  \Blockgraph
  \end{center}
\caption{Regions in a block: $R_5$ is the unbounded region}\label{BlockgraphF}
\end{figure}

The Seifert graph is a bipartite graph because the projection surface is orientable, hence every circuit in the graph must have an even length. In particular, we can choose a sign for an arbitrary vertex, and extend this labelling to the other vertices in an alternating fashion, when moving along the edges. We also have, for each edge $e$ in $B$, its corresponding sign $\epsilon (e)$ (if the original diagram is homogeneous, this sign is constant in the block). We define $E_{ij}$ as the set of edges in $\partial R_i \cap \partial R_j$ with the sign arrangement shown in Figure~\ref{enfrentadosF}.
\begin{figure}[ht!]
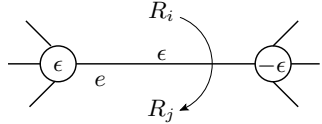

\labellist
\small
\pinlabel {$R_j$} at 195 0
\pinlabel {$e$} at 118 38
\pinlabel {$\epsilon$} at 65 56
\pinlabel {$\epsilon$} at 195 71
\pinlabel {$-\epsilon$} at 333 56
\pinlabel {$R_i$} at 195 125
\endlabellist
  \begin{center}
  \enfrentados
  \end{center}
\caption{The edge $e$ belongs to $E_{ij}$ with this arrangement of signs} \label{enfrentadosF}
\end{figure}

It turns out that
$$
{\rm lk}(g_i,g_i^+)=\frac{1}{2}\sum_{e \in \partial R_i} -\epsilon (e)
$$
and
$$
{\rm lk}(g_i,g_j^+)=\sum_{e \in E_{ij}} \epsilon (e).
$$

In particular, if the block is homogeneous, let say with sign $\epsilon$, then
$$
{\rm lk}(g_i,g_i^+)=-\epsilon k_i
$$
where $2k_i$ is the number of edges in $\partial R_i$, and
$$
{\rm lk}(g_i,g_j^+)=\epsilon |E_{ij}|.
$$
In other words, ${\rm lk}(g_i,g_j^+)$ is the number (with sign $\epsilon$) of the edges $e$ in the frontier of the regions $R_i$ and $R_j$, such that one leaves the $-\epsilon$ signed vertex on the left when going from $R_i$ to $R_j$ through the edge $e$ (see Figure~\ref{enfrentadosF}).

As an example, we display the Seifert matrix associated to the above graph, assuming that the top left vertex is labelled with sign $\epsilon$ (see Figure \ref{BlockgraphWithSignsF}):
$$
\begin{array}{rl}
\ & \hspace{0.2cm} \begin{array}{ccccc} & g_1^+ & \hspace{0.05cm} g_2^+ & \hspace{0.1cm} g_3^+ & \hspace{0.2cm} g_4^+ \end{array} \\
\vspace{-0.2cm}
\ & \hspace{0.2cm} \begin{array}{ccccc} &       &                       &                      & \end{array} \\
\begin{array}{c}
g_1 \\
g_2 \\
g_3 \\
g_4
\end{array} &
\left(
\begin{array}{cccc}
-3\epsilon & \epsilon   & \epsilon    & 0          \\
\epsilon   & -2\epsilon & 0           & \epsilon   \\
0          & \epsilon   & -3\epsilon  & 0          \\
\epsilon   & 0          & \epsilon    & -2\epsilon
\end{array}
\right)
\end{array}
$$

\begin{figure}[ht!]
\labellist
\small
\pinlabel {$R_1$} at 30 140
\pinlabel {$R_2$} at 108 108
\pinlabel {$R_3$} at 200 105
\pinlabel {$R_4$} at 105 35
\pinlabel {$R_5$} at -50 150
\pinlabel {$\epsilon$} at 4 211
\pinlabel {$-\epsilon$} at 188 211
\pinlabel {$-\epsilon$} at 58 162
\pinlabel {$\epsilon$} at 149 162
\pinlabel {$\epsilon$} at 250 148
\pinlabel {$\epsilon$} at 45 78
\pinlabel {$-\epsilon$} at 180 65
\pinlabel {$-\epsilon$} at 255 60
\pinlabel {$-\epsilon$} at 0 -10
\pinlabel {$\epsilon$} at 190 -10
\endlabellist
  \begin{center}
  \Blockgraph
  \end{center}
\caption{A homogeneous block (all the edges have sign $\epsilon$)}\label{BlockgraphWithSignsF}
\end{figure}

The sets $E_{ij}$'s satisfy two properties, which will play later a central role, specially in Theorem~\ref{MatrixPropertiesTheorem}:
\begin{enumerate}
    \item If $e \in \partial R_i \cap \partial R_j$, then $ e \in E_{ij} \Leftrightarrow e \notin E_{ji}$, and in particular $|E_{ij}|+|E_{ji}|$ is the cardinal of the edges in $\partial R_i \cap \partial R_j$.
    \item Consider two consecutive edges $e$ and $f$ in the boundary of a certain region $R_i$, that separate $R_i$ from $R_j$ and $R_k$ respectively, with possibly $j=k$. Suppose that both edges have the same sign, which is the case if we have a homogeneous graph. Then $e \in E_{ij} \Leftrightarrow f \in E_{ki}$.
\begin{figure}[ht!]
\labellist
\small
\pinlabel {$R_j$} at 112 75
\pinlabel {$R_k$} at 259 75
\pinlabel {$e$} at 112 22
\pinlabel {$f$} at 259 22
\pinlabel {$R_i$} at 184 -10
\endlabellist
  \begin{center}
  \consecutiveedges
  \end{center}
\caption{Consecutive edges in the boundary of a region $R_i$}\label{consecutiveedgesF}
\end{figure}
\end{enumerate}

\begin{remark}
Note that, for a homogeneous block with sign $\epsilon$, the sum of two transposed elements in the corresponding Seifert matrix gives
$$
{\rm lk}(g_i,g_j^+)+{\rm lk}(g_j,g_i^+)=\epsilon |E_{ij}|+\epsilon |E_{ji}|=\epsilon \ |\partial R_i \cap \partial R_j|.
$$
\end{remark}

\subsection{The directed dual graph}\label{dualS}
A description of the Seifert matrix corresponding to a homogeneous block can be better understood as a certain matrix of degrees for the oriented dual graph. To construct the directed dual graph we draw a vertex $v_i$ in the region $R_i$, including a vertex $v_{r+1}$ for the unbounded region $R_{r+1}$, and for each edge $e$ in $\partial R_i \cap \partial R_j$ we draw an edge $\bar{e}$ joining $v_i$ and $v_j$, the edge $\bar{e}$ intersecting the original graph only in $e$. Moreover, the edge $\bar{e}$ is oriented from $v_i$ to $v_j$ if (and only if) $e \in E_{ij}$. An example is exhibited in Figure \ref{dualchicoF}, assuming the sign $\epsilon =+1$ for all the edges and for the top left vertex.
\begin{figure}[ht!]
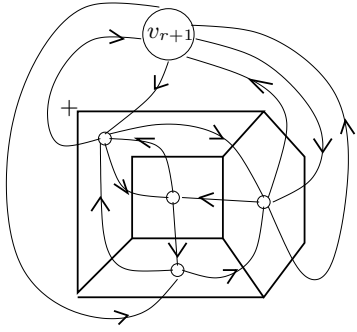

\labellist
\small
\pinlabel {$+$} at 63 217
\pinlabel {$v_{r+1}$} at 164 289
\endlabellist
  \begin{center}
  \dualchico
  \end{center}
\caption{The directed dual graph, all the edges assumed to have sign $+1$}\label{dualchicoF}
\end{figure}

Note that the edges incident at any vertex have alternative orientations, which is equivalent to the second property of the sets $E_{ij}$'s. In particular the degrees of the vertices are even numbers.

We now define $m_{ii}=-\epsilon \ {\rm deg}_i$ and $m_{ij}=\epsilon \ {\rm deg}_{ij}$ where ${\rm deg }_i$ is the number of edges leaving (or going to) $v_i$ and ${\rm deg}_{ij}$ is the number of edges from $v_i$ to $v_j$. It turns out that the matrix $(m_{ij})_{1\leq i,j \leq r+1}$ has determinant zero, and we obtain the Seifert matrix of the block by just deleting its last row and column. In the above example, for $\epsilon =+1$, we would have
$$
\left(
\begin{array}{ccccc}
-3 &  1 &  1 &  0 &  1 \\
 1 & -2 &  0 &  1 &  0 \\
 0 &  1 & -3 &  0 &  2 \\
 1 &  0 &  1 & -2 &  0 \\
 1 &  0 &  1 &  1 & -3
 \end{array}
\right).
$$

One should note that this is essentially what Proposition 13.21 in \cite{Burde} states, where the adjective special is applied to a diagram if the union of the black regions (assumed a chessboard colouring in which the unbounded region is white) is the image of a Seifert surface under the projection that defines the diagram.

\subsection{Properties of the matrix for a homogeneous block}\label{propertiesS}

Let $\epsilon$ be a sign, $+1$ or $-1$. A square matrix $A$ is said to be $\epsilon$-signed if its diagonal elements have sign $-\epsilon$ (in particular they do not vanish) and the elements out of the diagonal are zero or have sign 
$\epsilon$. The matrix $A$ is said to be row-dominant if for any row $i$ we have $|a_{ii}|\geq \sum_{j\not= i}|a_{ij}|$. The matrix $A$ is said to be strictly ascending row-dominant (abbreviated, {\it sard}) if $A$ is row-dominant and, in addition, there is an order of its rows $i_1 < \dots < i_r$ such that $|a_{i_ri_r}|>0$ and for any $k \in \{ 1, \dots , r-1 \}$ we have that $|a_{i_ki_k}| > \sum_{j\not= i_1, i_2, \dots , i_k}|a_{i_kj}|$.

The following matrix can be seen to be $(+)$-signed and sard choosing the order $3,1,2$ for its rows (note that the condition $|a_{i_ri_r}|>0$ is for sure if $A$ is $\epsilon$-signed):
$$
\left(
\begin{array}{ccc}
-3 & 0 & 3 \\
0 & -2 & 1 \\
1 & 0 & -2
\end{array}
\right)
$$
\begin{theorem}\label{MatrixPropertiesTheorem}
Let $B$ be a homogeneous block with sign $\epsilon$. Then there exists a basis of $H_1(B)$ such that the associated Seifert matrix $M$ is $\epsilon$-signed and sard.
\end{theorem}
\begin{proof}
Consider the basis of $H_1(B)$ given by the counterclockwise oriented cycles $\{ g_1, \dots , g_r\}$, boundaries of the bounded regions $R_i$ of $B$. Then the Seifert matrix $M=(a_{ij})_{1\leq i,j \leq r}$ is obviously $\epsilon$-signed since $a_{ii}={\rm lk}(g_i,g_i^+)=-\epsilon k_i$ where $k_i$ is half the number of edges in the boundary of $R_i$, and $a_{ij}={\rm lk}(g_i,g_j^+)=\epsilon |E_{ij}|$ if $i \not= j$. To see that $A$ is row-dominant note that $|a_{ii}|=k_i$, and on the other hand $\sum_{j\not= i}|a_{ij}|=\sum_{j\not= i}|E_{ij}|\leq k_i$, the inequality by the second property of the sets $E_{ij}$'s.

We finally check that the matrix $M$ is sard, by finding an order $i_1 < \dots < i_r$ for its rows such that $|a_{i_ki_k}| > \sum_{j\not= i_1, i_2, \dots , i_k}|a_{i_kj}|$ for any $k \in \{ 1, \dots , r-1 \}$. By the second property of the sets $E_{ij}$'s there is always a bounded region $R_i$ such that $E_{i,r+1}\not= \emptyset$. The corresponding row is chosen to be the first one in this order, that is, $i_1=i$. Note that, since $|a_{ii}| \geq \sum_{1\leq j \leq r+1, j\not= i}|E_{ij}|$ and $E_{i,r+1}\not= \emptyset$, it follows that $|a_{ii}| > \sum_{j\not= i}|a_{ij}|$. Now, when we delete the $i$-th row and column, the remaining matrix corresponds to the graph that remains after deleting the region $R_i$ (precisely, deleting the intersection between $R_i$ and $R_{r+1}$). The region can be also taken in such a way that the remaining graph is still a homogeneous block, hence the repetition of this process provides the wanted order for the rows of $M$.
\end{proof}

\subsection{The determinant for a homogeneous block}\label{determinantS}
In this section we will prove that, given a block with sign $\epsilon$ and rank $r$, the determinant of the corresponding submatrix is non-zero, and its sign is equal to $(-\epsilon)^r$. To see this we just need a final result, purely algebraic, due to Murasugi (see \cite{Murasugi}, section 2). For the convenience of the reader, we reproduce here its proof in a slightly different way:

\begin{theorem}\label{DeterminantTheorem} {\rm (Murasugi)}
Let $A$ be a square matrix of order $r$, $\epsilon$-signed and sard. Then ${\rm det} (A) <0$ if $\epsilon =+1$ and $r$ is odd, and ${\rm det} (A)>0$ otherwise. In other words, ${\rm det}(A)$ does not vanish and has sign $(-\epsilon )^r$.
\end{theorem}
\begin{proof}
By induction on $r$. The case $r=1$ (odd) is trivial; for $A=(a)$ we have ${\rm det}(A)=a$, and the result follows from the fact that $A$ is $\epsilon$-signed.

Assume now the statement for cases $1$ to $r-1$, and consider the case $r$. Since $A$ is sard, there is an order $i_1 < \dots < i_r$ of the rows such that for any $k \in \{ 1, \dots , r-1 \}$ we have $|a_{i_ki_k}|>\sum_{j \not = i_1, i_2, \dots , i_k}|a_{i_kj}|$. In particular, we have that
$$
a_{i_1i_1}=-\sum_{j\not= i_1}a_{i_1j}-\lambda
$$
with $\lambda\not= 0$ and sign $\epsilon$. We now develop the determinant by the $i_1$-row, obtaining
$$
{\rm det}(A)
=
{\rm det}\left( \begin{array}{c} \dots \\ a_{i_11} \ \dots \ a_{i_1i_1} \ \dots \ a_{i_1r} \\ \dots \end{array}\right)
=x-\lambda y
$$
where
$$
x={\rm det}\left( \begin{array}{c} \dots \\ a_{i_11} \ \dots \ -\sum_{j\not= i_1}a_{i_1j} \ \dots \ a_{i_1r} \\ \dots \end{array}\right)
$$
and $y$ is the determinant of the square matrix of order $r-1$, obtained by deleting the $i_1$-th row and column. Since this matrix is also $\epsilon$-signed and sard, by induction $y=(-\epsilon)^{r-1}|y|\not= 0$. Moreover, if each $a_{i_1j}=0$ for $j\not= i_1$ then $x=0$ obviously; otherwise it is a square matrix of order $r$, $\epsilon$-signed and row-dominant, and by Lemma~\ref{auxiliar} we have that $x=0$ or has sign $(-\epsilon )^r$. Then
$$
{\rm det}(A)=x-\lambda y=(-\epsilon )^r|x|-\epsilon |\lambda |(-\epsilon )^{r-1}|y|=(-\epsilon )^r(|x|+|\lambda | |y|)
$$
and the result follows since $|\lambda | >0, |y|>0$ and $|x|\geq 0$.
\end{proof}

\begin{lemma}\label{auxiliar}
Let $A$ be a square matrix of order $r$, $\epsilon$-signed and row-dominant. Then ${\rm det} (A)\leq 0$ if $\epsilon =+1$ and $r$ is odd, and ${\rm det} (A)\geq 0$ otherwise.
\end{lemma}
\begin{proof}
For technical reasons in the induction argument, we will prove this result for a slightly wider category of matrices, the weak $\epsilon$-signed and row-dominant matrices. For this matrices the condition of being $\epsilon$-signed is relaxed for allowing zeros in the diagonal.
  
We proceed by induction on $r$. The case $r=1$ is trivial. Assume now the statement for cases $1$ to $r-1$, and consider the case $r$. Since $A$ is weak $\epsilon$-signed and row-dominant, each diagonal element of $A$ can be written as $a_{ii}=-\sum _{j\not= i} a_{ij} -\lambda _i$ with $\lambda _i =0$ or with sign $\epsilon$.

Let $A_i$ be the same matrix as $A$ except for possibly the first $i$ elements of its diagonal, where we consider $a_{ii}+\lambda _i$ instead of just $a_{ii}$. Let $A_0=A$. It turns out that
$$
{\rm det}(A)={\rm det}(A_r)-\sum_{i=1}^r \lambda_i\ {\rm det}((A_{i-1})_i^i).
$$
where the notation $B_i^i$ is used to denote the matrix obtained from $B$ by deleting its $i$-th row and $i$-th column. This follows from the equalities ${\rm det}(A_{k-1})={\rm det}(A_{k})-\lambda_k{\rm det}((A_{k-1})_k^k)$, $k=1, \dots , r$.

Note that the determinant of $A_r$ is equal to zero, since the sum of all the elements of each row is zero. Moreover, each matrix $(A_{i-1})_i^i$ is also weak $\epsilon$-signed and row-dominant, and has order $r-1$. By induction, its determinant is zero or has sign $(-\epsilon )^{r-1}$. Since each $\lambda _i$ is zero or has sign $-\epsilon$, the result follows.

\end{proof}

As an application of the argument developed in this section, we prove the following result:

\vspace{0.3cm}

\noindent {\bf Claim.} {\it Let $L$ be an oriented link which has a special alternating diagram. Then the leading coefficient of $\nabla _L(z)$ is $\pm 1$ if and only if $L$ is the connected sum of $(2,q)$-torus links.} 

\begin{proof} Assume that $L$ is the connected sum of $(2,q)$-torus links. Since $\nabla_{L\sharp L'}(z)= \nabla_{L}(z)\nabla_{L'}(z)$, it is enough to show that the leading coefficient of $\nabla_{L}(z)$ is $\pm 1$ if $L$ is a $(2,q)$-torus link. The diagram shown in Figure~\ref{ClaimExtraF} (left), or its mirror image, is then a diagram of $L$. It has $q$ crossings, all with the same sign~$\epsilon$. The corresponding Seifert graph, shown on the right in Figure~\ref{ClaimExtraF}, is a homogeneous block $B$ with two vertices and $q$ edges, all of them with sign $\epsilon$.
\begin{figure}[ht!]
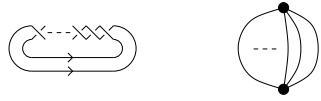

\labellist
\small
\endlabellist
  \begin{center}
  \ClaimExtra
  \end{center}
\caption{A $(2,q)$-torus link diagram and the corresponding Seifert graph}\label{ClaimExtraF}
\end{figure}

Following the process explained at the beginning of this section, we obtain the Seifert matrix $M=(m_{i,j})_{i,j=1, \dots , q-1}$ where $m_{i,i}=-\epsilon$ for $i=1, \dots , q-1$, $m_{i,i+1}=\epsilon$ for $i=1, \dots , q-2$, and $m_{i,j}=0$ otherwise. Then the leading coefficient of $\nabla_L(z)$ is $\epsilon ^{{\rm rk}(B)}|{\rm det}(M)|=\epsilon^{q-1}$ since ${\rm rk}(B)=1-v+e=1-2+q=q-1$.

\vspace{0.3cm}

Suppose now that the leading coefficient of $\nabla _L(z)$ is $\pm 1$, and $L$ has a special alternating diagram $D$. Then $D$ is the connected sum of diagrams $D_1, \dots , D_r$ where each $D_i$ is a diagram (of a link $L_i$) such that its Seifert graph has only one (homogeneous) block (see Figure~\ref{SpecialAlternatingDiagramF}).
\begin{figure}[ht!]
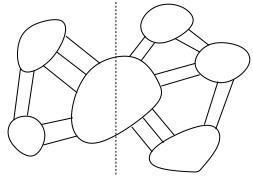

\labellist
\small
\endlabellist
  \begin{center}
  \SpecialAlternatingDiagram
  \end{center}
\caption{A special alternating diagram}\label{SpecialAlternatingDiagramF}
\end{figure}

Clearly, $L=\sharp_{i=1}^rL_i$. Since $\nabla_L(z)=\prod_{i=1}^r\nabla_{L_i}(z)$ and $\nabla_L(z) \in Z[z^{\pm 1}]$, the leading coefficient of each $\nabla_{L_i}(z)$ is $\pm 1$. Hence it is enough to prove that $L$ is a $(2,q)$-torus link assuming that the leading coefficient of $\nabla_L(z)$ is $\pm 1$, and $L$ has a diagram $D$ whose associated Seifert graph is a homogeneous block $B$, let say with sign $\epsilon$.

We will prove that $B$ has the form of the graph shown on the right in Figure~\ref{ClaimExtraF}, by induction on the number of edges of $B$. With this aim, we order the $r$ bounded regions of $B$ as in the proof of Theorem~\ref{MatrixPropertiesTheorem}. The corresponding Seifert matrix $A$ is then $\epsilon$-signed and sard, and by the proof of Theorem~\ref{DeterminantTheorem}, we have
$$
{\rm det}(A)=(-\epsilon )^r(|x|+|\lambda ||y|)
$$
where $y={\rm det}(A_1^1)\not= 0$. Since the leading coefficient of $\nabla_L(z)$ is $\pm 1$, we have ${\rm det}(A)=\pm 1$; since $\lambda$ and $y$ are nonzero integers, we have $y={\rm det}(A_1^1)=\pm 1$.

Now, according to the proof of Theorem~\ref{MatrixPropertiesTheorem}, $A_1^1$ is the Seifert matrix associated to the diagram $D'$ whose Seifert graph is $B'=B\setminus(R_1 \cap R_{r+1})$, where $R_{r+1}$ is the unbounded region of $B$. Since $B'$ is still a homogeneous block, by induction we have that $B'$ has the form of the graph shown on the right in Figure~\ref{SpecialAlternatingDiagramF}, and $B$ adds a path connecting the two vertices of $B'$ in the unbounded region of $B'$, as shown in Figure~\ref{FormOfBF}. 
\begin{figure}[ht!]
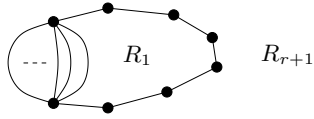

\labellist
\small
\pinlabel {$R_1$} at 260 115
\pinlabel {$R_{r+1}$} at 560 115
\endlabellist
  \begin{center}
  \FormOfB
  \end{center}
\caption{The graph $B$}\label{FormOfBF}
\end{figure}

Let $2k$ be the number of edges bounding $R_1$ in $B$. Then the original Seifert matrix is 
$$
A=
\left( 
	\begin{array}{c|c}
		k & \pm 1 \ 0 \ \cdots \ 0 \\ 
		\hline 
		\begin{array}{c} 0 \\ \vdots \\ 0 \end{array} & A_1^1 
	\end{array}
\right)
$$ 
or its transpose, in any case with determinant $\pm k$. Hence $k=1$ and the result follows.
\end{proof}

\begin{corollary} Let $L$ be an oriented homogeneous link. Then the leading coefficient of $\nabla _L(z)$ is $\pm 1$ if and only if $L$ is the Murasugi sum of connected sums of $(2,q)$-torus links.
\end{corollary}

\section{Homogeneous knots of genus one}\label{ClassificationS}
We finish the paper with a complete classification of the family of homogeneous knots of genus one. Let $D$ be a homogeneous diagram of a homogeneous knot $K$ of genus one. Let $F$ and $G$ be respectively the projection surface and the Seifert graph associated to the diagram $D$. We already know that the genus of $F$ is exactly the genus of the knot. Since $2g(F)+\mu -1 ={\rm rk }(G)$ and $K$ is a link with one component, we deduce that $G$ has rank two. Figure \ref{FHomogeneousGraphs} exhibits the two types of such graphs; they can be denoted by $G(a,b,c)$ and $G(m,k)$ respectively, where the absolute values of the integers $a,b,c,m,k$ are the numbers of corresponding edges, and their signs are the signs of these edges. 

\begin{figure}[ht!]
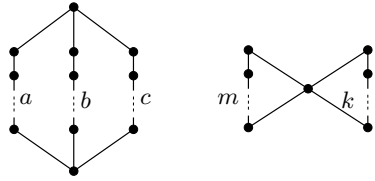

\labellist
\small
\pinlabel {$a$} at 25 115
\pinlabel {$b$} at 115 115
\pinlabel {$c$} at 205 115
\pinlabel {$m$} at 330 115
\pinlabel {$k$} at 510 115
\endlabellist
  \begin{center}
\GenusTwoHomogeneousGraphOneBlock \qquad \qquad \GenusTwoHomogeneousGraphTwoBlocks
 \end{center}
\caption{Homogeneous graphs with rank two: one and two blocks}\label{FHomogeneousGraphs}
\end{figure}

Note that these graphs could have some tails, but this would not affect to the knot type. Since $G(a,b,c)$ is homogeneous and has only one block, $a,b,c$ must have all the same sign; since $F$ is orientable, they have also the same parity. On the contrary, $G(m,k)$ has two blocks, hence $m$ and $k$ can have different signs, but both must be even because of the orientability. Note also that the second graph can be considered a degenerated form of the first one, with $b=0$.

In general, the Seifert graph does not determine the link where it comes from, although in the first case it does. In $G(a,b,c)$ there are exactly two trivalent vertices; the corresponding Seifert circles can be one inside the other, or separated. When viewed this in the sphere $S^2$ there is no difference, and the corresponding knot is the pretzel knot with diagram $P(a,b,c)$. Moreover, since $P(a,b,c)$ must be a knot, the numbers $a,b,c$ should be all odd, or exactly one of them should be even. It follows that all of them are odd.

Consider now the graph with two blocks $G(m,k)$. There is only one vertex with valence four, given the two possible configurations for the Seifert circles shown in Figure \ref{FConfigurations}.

\begin{figure}[ht!]
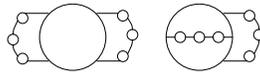

  \begin{center}
\ConfigurationsSeifertGraphsTwoBlocks
 \end{center}
\caption{Possible configurations of the Seifert circles for $G(m,k)$ (examples with $|m|=|k|=4$)}\label{FConfigurations}
\end{figure}

The first configuration corresponds to a link with three components, and the second one corresponds to a knot $K$ (Figure \ref{FOneBlock}). Moreover, the obtained knot $K$ is also a pretzel knot, given by the pretzel diagram $D(m,k)=P(m,\epsilon , \stackrel{|k|}{\dots}, \epsilon)$, where $m$ and $k$ are even integers and $\epsilon$ is the sign of $k$. For example, $D(4,-2)=P(4,-1,-1)$ is the example in Figure \ref{FOneBlock}, right.
\begin{figure}[ht!]
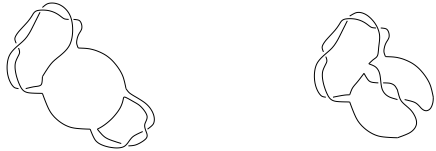

  \begin{center}
\OneBlockNoKnot \qquad \qquad \qquad \OneBlockYesKnot
 \end{center}
\caption{Links coming from rank two homogeneous graph with two blocks: only the second one is a knot ($D(4,-2)$ in the example)}\label{FOneBlock}
\end{figure}

What we have done is to prove the following result:

\begin{theorem}\label{classification}
A genus one knot is homogeneous if and only if it belongs to one of the two following classes of knots:
\begin{enumerate}
  \item Pretzel knots with diagram $P(a,b,c)$, where $a,b,c$ are odd integers with the same sign.
  \item Pretzel knots with diagram $D(m,k)=P(m,\epsilon , \stackrel{|k|}{\dots}, \epsilon)$, where $m$ and $k$ are non-zero even integers and $\epsilon =k/|k|$ is the sign of $k$.
\end{enumerate}
\end{theorem}

The given classification and some partial information of the Jones polynomial allow us to give another proof of the following result, due to Cromwell \cite{PeterPaper}:

\begin{corollary}\label{coro1}
Pretzel knots $P(p,-q,-r)$ with $3 \leq p \leq q \leq r$, all of them odd, are not homogeneous.
\end{corollary}

In the original proof, Cromwell calculated the Homfly polynomial $P(v,z)=\sum _{i=0}^r \alpha_i(v)z^i$ and checked that $\alpha _r(v)$ contains terms of both signs (\cite{PeterPaper}, Theorem~10). But, for homogeneous links, these coefficients are all non-negative or all non-positive, according to a result by Traczyk (\cite{PeterPaper}, Corollary~{4.3}).

\begin{proof}
We want to prove that the knot $K$ defined by a pretzel diagram $P(p,-q,-r)$ is not homogeneous. First note that $K$ has genus one, since the projection surface defined by the diagram $P(p,-q,-r)$ has Euler characteristic~$-1$, hence genus one, and $K$ is not the trivial knot; for example, according to Theorem 2, case (iv) (a) in \cite{pretzel}, the span of its Jones polynomial (with normalization $-t^{-\frac{1}{2}}-t^{\frac{1}{2}}$) is $p+q+r-{\rm min}\{ p, q-1 \}$, which is different from one since $3 \leq p \leq q,r$.

Now, the lowest degree and the coefficient of the highest degree term of the Jones polynomial tell us that $K$ does not belong to any of the two classes of homogeneous knots of genus one given by Theorem~\ref{classification}, as the following table shows:
$$
\begin{array}{|c|c|c|c|}
\hline
\multicolumn{2}{|c|}{\text{Knot diagram}} & \text{Lowest degree} & \begin{array}{c} \text{Coefficient of the} \\ \text{highest degree term} \end{array} \\
\hline
\hline
\multirow{2}{*}{$P(p,-q,-r)$} & 3 \leq p <q \leq r & 1/2         & -1 \\
\cline{2-4}
                              & 3 \leq p =q \leq r & -1/2        &    \\
\hline
\hline
\multirow{2}{*}{$P(a,b,c)$}   & 0 \leq a,b,c    & -3/2-a-b-c  &    \\
\cline{2-4}
                              & a,b,c \leq 0    & 1/2         & 1  \\
\hline
\hline
\multirow{4}{*}{$D(m,k)$}     & m,k>0           & -m-1/2      &    \\
\cline{2-4}
                              & m<0, k>0        & 1/2         & 1  \\
\cline{2-4}
                              & m>0, k<0        & k-m-1/2     &    \\
\cline{2-4}
                              & m,k<0           & k-1/2       &    \\
\hline
\end{array}
$$
\end{proof}

Note that the Conway polynomial together with the span of the Jones polynomial are not enough in order to prove Corollary~\ref{coro1}. According to the values displayed in the following table, we have for example that the knots defined by the diagrams $P(3,-45,-91)$ and $P(11,23,101)$ share Conway polynomial and the span of their Jones polynomials, and the same happens to the pair of knots defined by the diagrams $P(11,-15,-15)$ and $D(-4,26)$.
$$
\begin{array}{|c|c|c|c|}
\hline
\multicolumn{2}{|c|}{\text{Knot diagram}} & \begin{array}{c} \text{Conway polynomial } \\ 1+\lambda z^2 \text{ where } \lambda \text{ is } \end{array} & \text{Jones polynomial span} \\
\hline
\hline
\multirow{2}{*}{$P(p,-q,-r)$} & 3 \leq p <q \leq r & \multirow{2}{*}{$(qr-pq-pr+1)/4$}  & q+r \\
\cline{2-2}\cline{4-4}
                              & 3 \leq p =q \leq r &                                    & q+r+1 \\
\hline
\hline
\multirow{2}{*}{$P(a,b,c)$}   & 0 \leq a,b,c    & \multirow{2}{*}{$(ab+ac+bc+1)/4$}  & 1+a+b+c \\
\cline{2-2}\cline{4-4}
                              & a,b,c \leq 0    &                                    & 1-a-b-c \\
\hline
\hline
\multirow{4}{*}{$D(m,k)$}     & m,k>0           & \multirow{4}{*}{$mk/4$}            & 1+m+k \\
\cline{2-2}\cline{4-4}
                              & m<0, k>0        &                                    & k-m \\
\cline{2-2}\cline{4-4}
                              & m>0, k<0        &                                    & m-k \\
\cline{2-2}\cline{4-4}
                              & m,k<0           &                                    & 1-m-k \\
\hline
\end{array}
$$

We also have the following result (as above, the Jones polynomial of the pretzel links and their spans have been calculated following~\cite{pretzel}):
\begin{corollary}\label{coro2}
At least one of the extreme coefficients of the Jones polynomial of a homogeneous knot of genus one is $-1$.
\end{corollary}

Finally, we would like to remark that Stoimenow \cite{Stoimenow} has showed that a genus two homogeneous knot is alternating or positive. In \cite{Jong} Jong and Kishimoto have studied genus two positive knots extensively.

\section*{Acknowledgments}
I am grateful to Prof. Hugh R. Morton for several helpful comments on a previous version of this paper. I am also grateful to the referee for a number of helpful suggestions for the article's improvement. An overview of this work was presented at Knots in Poland III Conference; I am grateful to its organizers for their hospitality.

The author is partially supported by Spanish Project MTM2010-19355 and FEDER.

\ \\ 
\noindent Pedro M. Gonz\'alez Manch\'on \\
Department of Applied Mathematics \\
Ronda de Valencia 3 \\
EUITI-UPM \\
28012 Madrid (Spain)\\
\ \\
pedro.gmanchon@upm.es \\
http://gestion.euiti.upm.es/index/departamentos/matematicas/manchon/index.htm \\

\end{document}